\documentclass[11pt,twoside]{amsart}

\usepackage{amssymb, palatino,graphicx, hyperref}
\usepackage[normalem]{ulem}
\usepackage[all]{xy}
\usepackage{color}
\usepackage{lmodern}% http://ctan.org/pkg/lmodern
\usepackage{slantsc}% http://ctan.org/pkg/slantsc
\usepackage{ifthen}

\if@titlepage
  \newenvironment{resume}{%
      \titlepage
      \null\vfil
      \@beginparpenalty\@lowpenalty
        \bfseries \resumename
                \@endparpenalty\@M
      \end{center}}%
     {\par\vfil\null\endtitlepage}
\else
  
\fi

\newcommand\resumename{R\'esum\'e}

\graphicspath{{./}{./Figures/}}

\newtheorem{dfn}{Definition}[section]
\newtheorem{thm}[dfn]{Theorem}%[section]
\newtheorem{lem}[dfn]{Lemma}%[section]
\newtheorem{prop}[dfn]{Proposition}%[section]
\newtheorem{rem}[dfn]{Remark}%[section]

\def\inv{^{-1}}

\def\P{\ensuremath{\mathbb P}}
\def\Q{\ensuremath{\mathbb Q}}
\def\R{\ensuremath{\mathbb R}}
\def\Z{\ensuremath{\mathbb Z}}

\def\C{\ensuremath{\mathbb C}}

\def\a{\mathfrak{a}}
\def\so{\mathfrak{so}}

\def\CP1{\ensuremath{\mathbb C \mathbb P^1}}

\def\Pn-1{\ensuremath{\P^{n-1}}}
\def\Ea{\E_{\alpha}}

\def\cf{{\em cf. }}

\newcommand{\mysmallmatrix}[1]{\ensuremath{
                            \bigl[ \begin{smallmatrix}
#1
\end{smallmatrix} \bigr]}}

\newcommand{\applic}[5]{
  \ensuremath{\begin{array}[t]{c@{}c@{}c@{}l}
           #1:\ & #2 & \rightarrow & #3 \\
               & #4 & \mapsto     & #5
         \end{array}
        }
}

\newcommand{\applicnn}[4]{
  \ensuremath{\begin{array}[t]{c@{}c@{}l}
           #1 & \longrightarrow & #2 \\
           #3 & \longmapsto     & #4
         \end{array}
        }
                        }

\newcommand{\setofst}[2]
{ \ensuremath{ \left\{\, #1\ \left\vert\ #2 \right.\right\}
            }
}

\newcommand{\bvect}[1]{ {\scriptstyle{
\left[
\begin{array}{c}
#1
\end{array}
\right]               }}}

\newcommand{\horiz}{\ensuremath $\rotatebox{90}{$\vert$}$}

\definecolor{red}{rgb}{.6,0,0}
\definecolor{green}{rgb}{0,.6,0}
\definecolor{darkgreen}{rgb}{0,0.3,0}
\definecolor{purple}{rgb}{0.5,0,0.5}
\definecolor{darkblue}{rgb}{0,0,0.7}
\definecolor{greenblue}{rgb}{0,0.4,0.5}
\definecolor{yellow}{rgb}{0.8,0.5,0}

\newboolean{draft}

\newboolean{draftarxiv}

\newboolean{craft}

\newcommand{\cmt}[1]
{\ifthenelse {\boolean{draft}}
{}{}}

\newcommand{\rev}[1]
{\ifthenelse {\boolean{draft}}
{{\color{green} #1}}
{#1}}

\newcommand{\newbb}[1]
{\ifthenelse {\boolean{craft}}
{{\color{darkblue} #1}}
{#1}}

\newcommand{\newbbb}[1]
{\ifthenelse {\boolean{draft}}
{{\color{greenblue} #1}}
{#1}}

\newcommand{\nopost}[1]
{\ifthenelse {\boolean{draft}}
{{\color{cyan} #1}}
{}}

\newcommand{\maynopost}[1]
{\ifthenelse {\boolean{craft}}
{{\color{purple} #1}}
{}}

\newcommand{\margincmt}[1]
{\ifthenelse {\boolean{draft}}
{\marginpar{{\sc \tiny \color{red} #1}}}
{}}
\newcommand{\inred}[1]
{\ifthenelse{\boolean{draft}}{{\color{red} #1}}{#1}}

\newcommand{\new}[1]
{\ifthenelse {\boolean{draft}}
{{\color{green} #1}}
{#1}}

\newcommand{\neww}[1]
{\ifthenelse {\boolean{draft}}
{{\color{darkgreen} #1}}
{#1}}

\newcommand{\newb}[1]
{\ifthenelse {\boolean{draftarxiv}}
{{\color{blue} #1}}
{#1}}

\newcommand{\del}[1]
{\ifthenelse {\boolean{draft}}
{{\color{red} #1}}
{}}

\newboolean{details_on}

\newcommand{\details}[1]
{\ifthenelse {\boolean{details_on}}
{{\color{darkgreen} \tiny #1}}
{}}

\title{Foliations modeling nonrational simplicial toric varieties}
\author{Fiammetta Battaglia and Dan Zaffran}
\thanks{The second named author gratefully acknowledges the support of the Korea Advanced Institute of Science and Technology and of the Basic Science Research Program through the National Research Foundation of Korea (NRF) funded by the Korean Ministry of Education, Science and Technology (Grant No. 2010-0005879)
}

\newcommand{\vu}{\ensuremath{ \underline{u} } }

\newcommand{\vv}{\ensuremath{ \underline{v} } }
\newcommand{\vx}{\ensuremath{ \underline{x} } }
\newcommand{\vy}{\ensuremath{ \underline{y} } }
\newcommand\vt{\ensuremath{\underline{t}}}
\newcommand\vs{\ensuremath{\underline{s}}}
\newcommand\vz{\ensuremath{\underline{z}}}

\newcommand\vgamma{\ensuremath{\underline{\gamma}}}

\newcommand{\bT}{\ensuremath{\mathcal{T} } } %{ \pmb{\mathcal{T}} } }
\newcommand\bE{\ensuremath{\mathcal E}} %{\pmb{\mathcal E}}}
\newcommand\E{\ensuremath{\mathcal E}}
\newcommand{\calF}{\ensuremath{ \mathcal{F} } }
\newcommand{\calG}{\ensuremath{ \mathcal{G} } }

\newcommand{\calC}{\ensuremath{ \mathcal{C} } }

\newcommand{\calH}{\ensuremath{ \mathcal{H} } }

\newcommand{\lr}[1]{\ensuremath{ \Lambda^\R_{#1} } }

\newcommand\expo[1]{\ensuremath{e^{2\pi i #1}}}
\DeclareMathOperator{\Rel}{Rel}

\setboolean{draft}{true}
\setboolean{details_on}{true}
\setboolean{draft}{false}
\setboolean{details_on}{false}

\setboolean{craft}{true}
\setboolean{craft}{false}

\setboolean{draftarxiv}{false}

\begin{document}
\maketitle
\begin{abstract}
We establish a correspondence between simplicial fans, not necessarily rational, and certain foliated compact complex manifolds called LVMB-manifolds. In the rational case, Meersseman and Verjovsky have shown that the leaf space is the usual toric variety.
We compute the basic Betti numbers of the foliation for shellable fans. When the fan is in particular polytopal,
we prove that the basic cohomology of the foliation is generated in degree two. 
\new{We give} evidence that the rich interplay between convex and algebraic geometries \new{embodied by toric varieties} carries over to our nonrational construction. 
In fact, \new{our approach unifies} rational and nonrational cases.\end{abstract}

\section{Introduction}
\subsection*{Rational convex polytopes and toric varieties}
The correspondence between rational convex polytopes and projective toric varieties is well-known. 
It pertains to several fields, including combinatorics, convex geometry, symplectic geometry, algebraic geometry.

Within this picture, simple rational polytopes correspond to toric varieties that are rationally smooth 
(i.e., having at most orbifold singularities). We will only consider this restricted correspondence, which has long been known to provide fruitful links 
between the fields listed above \cite{D,Gr,Kho,Tei,S}.

On the other hand, simple 
polytopes come in continuous families (by perturbing the facets' directions), 
whereas toric varieties, dubbed ``frigid crystals'' 
in \cite{Da}, do not. The reason is that no toric variety corresponds to a nonrational polytope. 
A solution to this problem was given by Prato in \cite{p}, by introducing a generalization of toric orbifolds 
which are non Hausdorff when the polytope is nonrational.

Relying on works by Meersseman and Verjovsky \cite{M,MV}, and by Prato \cite{p}
(and also on \cite{Bos,LN,ldm,CZ}),
we take a new approach by realizing the toric space corresponding to any simple convex 
polytope as the leaf space of a smooth foliation. This simultaneously provides an object in the nonrational case and removes 
all singularities. Even though we are 
in principle interested in the leaf space, we lift all statements and proofs to the level of the foliation, 
where everything is smooth and Hausdorff. 
\subsection*{A generalized correspondence.}
A more accurate formulation of the above correspondence is made in terms of fans: 
recall that to any convex polytope $P$, we can associate its normal fan 
$\Delta$, which is complete and polytopal. 
The map $P \mapsto X$ taking a rational simple convex polytope $P$ 
to a toric variety $X$ factors out as $P\mapsto\Delta\mapsto X$.
The first map is non-injective and the second one
is the classical one-to-one correspondence between 
complete simplicial rational polytopal fans 
and rationally smooth projective toric varieties (these objects beeing seen up to isomorphism).

We propose a four-way generalization of the correspondence $\Delta \mapsto X$. 
The roles of varieties and fans will be played, respectively, by  
{\em leaf spaces of} foliations on so-called LVMB-manifolds 
and suitable {\em triangulated vector configurations}.

Notice that a rational fan $\Delta$ determines implicitly a vector configuration $V$ on a lattice. 
Namely, $V$ is the set of primitive generators of the fan's rays. We will stress the importance of $V$ and its 
generalizations, to which we give the main role on the convex geometric side. 

The four generalizations are, in increasing order of importance from our point of view: 
\begin{enumerate}
\item We allow non-polytopal fans. This translates into using nonregular triangulations (\cf \ref{preliminaries}). 
\item We allow orbifold multiplicities. This amounts to taking nonprimitive generators on rays.
\item We have more generators than rays: some vectors in the configuration may not correspond to a ray. 
\item We do not require rationality of the configuration. This means dropping the closedness condition of the lattice.
\end{enumerate}
Each of these generalizations has already appeared in the literature. 
We will simultaneously refine, generalize or desingularize several known constructions, 
thus giving a unified picture of the smooth, orbifold and nonrational cases. 
\subsection*{Related works}
We give a simplified account of earlier constructions. Each starts from a convex-geometric object, 
and (using Gale duality) defines an algebraic- or  complex-geometric object. 
From now on, all fans (resp. polytopes) are simplicial (resp. simple). 
\subsubsection*{A rational fan.}
To each rational simplicial fan in a vector space $L\otimes_{\Z}\R$, with $L$ a lattice, 
there corresponds a rationally smooth toric variety $X$. 
\subsubsection*{A Delzant polytope (necessarily rational).}
On the symplectic side Delzant proves the existence of a unique symplectic toric manifold in correspondence 
to each Delzant polytope, i.e., whose normal fan satisfies suitable integrality conditions \cite{D}.
\subsubsection*{A rational polytope and multiplicities attached to facets 
 (equivalently, a rational polytopal fan with multiplicities attached to rays, 
 and a certain height function).}

\begin{enumerate}
\item {\em Symplectic orbifolds}.
Lerman and Tolman generalize Delzant's theorem to the class of symplectic toric orbifolds,
by allowing any rational convex simple polytope and rays generators that are 
not primitive \cite{lt}.\\ 
\item {\em Generalized Calabi-Eckmann fibrations I}.
Meersseman and Verjovsky prove in \cite{MV} that toric varieties and ``toric varieties with orbifold multiplicities'' 
can be viewed as leaf spaces of the foliations on {\em rational LVM-manifolds} \cite{M} (see below).
\end{enumerate}

\subsubsection*{A nonrational polytope, a quasilattice, and rays generators.} 
Prato generalizes the Delzant procedure to any simple convex polytope.
A key point is to replace the lattice with a {\em quasilattice}, i.e., 
a $\Z$-module in a vector space, generated by a finite spanning set. 
For a given simple convex polytope, different choices of a quasilattice and 
rays generators contained therein yield a family of symplectic spaces called {\em quasifolds}. 
When the polytope is rational, this family strictly contains the cases above.
When the quasilattice is not a lattice the corresponding spaces are non Hausdorff: 
nonrationality forces quotient singularities of finite type to become of ``infinite type'' \cite{p}. 
We shall refer to these spaces as {\em toric quasifolds}.
\subsubsection*{A nonrational polytope} 
Generalizing the works of L\'opez de Medrano and Verjovksy \cite{ldm}, and of L\oe b and Nicolau \cite{LN}, Meersseman \cite{M} constructs the so-called LVM-manifolds. It is formed by the compact complex foliated manifolds 
corresponding to our polytopal case. 
To each such manifold $N$, he associates a non necessarily rational polytope $P$; he establishes a one-to-one correspondence between the combinatorial type of $P$ and $N$ up to 
deformation \cite[Th.~13]{M}.
\subsubsection*{A(n implicit) nonrational fan} 
Generalizing Meersseman's construction, Bosio constructs the family 
of so-called LVMB-manifolds that we consider here (\cf \cite{Bos} and the interpretation in \cite{CZ}). 
\subsubsection*{A stacky fan (equivalently, a rational fan with multiplicities attached to rays).} 
\begin{enumerate}
\item {\em Generalized Calabi-Eckmann fibrations II}.
Tambour constructs and studies certain LVMB-manifolds in \cite{Tam}. 
He discusses the relationship with toric varieties. \\
\item {\em Stacks}. 
Another approach for handling the orbifold structure, and turning orbifolds into smooth objects, is to use stacks. 
We refer to Iwanari's article \cite{I} for this point of view, initiated by Borisov-Chen-Smith in \cite{BCS}. 
\end{enumerate}
\subsubsection*{A nonrational fan} 
Panov and Ustinovsky construct complex structures on moment-angle manifolds and their quotients by real tori in \cite{PU}. Under rationality assumptions, they discuss the relationship with toric varieties.

\medskip

Since we posted this article on the arXiv, several related results have appeared: Ishida \cite{ishida} discovered an interesting group-theoretic characterization of a class of manifolds strictly containing LVMB-manifolds. Ustinovsky showed in \cite{U} that Ishida's manifolds 
coincide with that of \cite{PU}. Another recent generalization was described 
by Battisti and Oeljeklaus \cite{BO}. 
Finally in the note \cite{KLMV} Katzarkov, Lupercio, Meersseman and Verjovsky investigate a different approach
to defining simplicial toric varieties in a nonrational setting. 
Their main technical tool is an extension of LVM theory to the nonrational case (i.e. they assume fans to be polytopal; \cf our Sect.~\ref{Reg Triang and Kahler}),
but they consider the leaf space from the non-commutative and diffeological (\cf \cite{IZ}) 
viewpoints.
\subsection*{Summary of our results}
We propose to encode all of the convex-geometric data needed for the construction of a 
space $X$ ---that is, a fan, a choice of a point on each ray and a (quasi)lattice
containing each of those points--- in a unique and well-studied object (\cf  \cite{DL-R-S}): a triangulated vector configuration $(V,\bT)$. Notice that a nonrational fan is not sufficient to determine a unique $X$. We develop a framework in which it is possible to obtain nonrational toric varieties
by means of LVMB-manifolds. In fact, we construct a well-defined map from $(V,\bT)$ to a complex-geometric object $X$ (the leaf space of an LVMB-manifold).  

We define in Sect.~\ref{vectorconfiguration} two integers $a$ and $b$ that are quantitative measures of the nonrationality of the configuration (the integer $a$ was defined in \cite{M}).
In correspondence to the configuration, we construct an LVMB-manifold $N$, endowed with a smooth 
holomorphic foliation $\calF$ whose topology depends on $a$ and $b$. 
The leaf space $X$ can be, in increasing generality, a smooth toric variety, a toric orbifold or a toric quasifold. 
In the latter cases, we see our smoothly foliated manifold $N$ as a desingularisation of the space $X$. In Sect.~\ref{examples} we include several fully worked out examples.

Beyond the introduction of $(V,\bT)$ as the main convex-geometric object, 
our main result is that, 
in both rational and nonrational cases, 
the cohomological study can be lifted to the foliation by using basic cohomology. In the case of a shellable fan, we compute the basic Betti numbers of $(N,\calF)$. 
In particular, we show that they only depend on the combinatorial type of the fan (Th.~\ref{prop betti numbers}).
When the fan is polytopal we prove that the basic cohomology algebra of $(N,\calF)$ is generated in degree two (Th.~\ref{generated in dim two}). 
In Sect.~\ref{Stanley} we show that our framework handles Stanley's proof of the necessity part
of the g-theorem, by applying El Kacimi's basic version
of the hard Lefschetz theorem \cite[3.4.7]{EK}.
Finally, in Sect.~\ref{Reg Triang and Kahler} we illustrate some specific features of the polytopal case. 

Throughout, we try to delineate the combinatorial, topological, and convex geometric aspects, 
each of which being of independent interest. On the convex-geometric side, we emphasize 
the relevance of methods such as triangulations, shellings and Gale duality. On the complex-geometric side, we explore and extend toric methods, giving evidence that at least part of the technology 
available with toric varieties carries over to our foliated model, 
which makes no distinction between rational and nonrational cases.  
LVMB-manifolds thus establish a tight link between convex geometry and complex geometry, and may also 
contribute to a more geometric understanding of the nonsimplicial nonrational case. 

\bigskip

\noindent {\small Acknowledgement: We would like to thank Dirk T\"oben for helpful conversations.}

\newpage

\section{Construction}
\subsection{Triangulated configurations}   
Let $E$ be an \R-vector space of dimension $d$.
\subsubsection{Vector configurations}\label{vectorconfiguration} A {\em vector configuration} 
$V=(v_1,\dots,v_n)$ is a finite, ordered list of vectors, allowing repetitions. 
We will assume that $\text{Span}_{\R}\{v_1,\dots,v_n\}=E$. 

Consider the space of linear relations among $v_{1},\dots, v_{n}$ 
$$\Rel(V):=\Big\{ c \in \R^{n} \Big\vert \sum_{1\leq j\leq n} c_{j} v_{j}=0\Big\},$$
which has dimension $n-d$.
We say that a real subspace of $\R^{n}$ is {\em rational} when it admits 
a real basis of vectors in $\Q^{n}$ (equivalently, $\Z^{n}$)
We define $a(V)$ as the 
dimension of the largest rational space contained in $\Rel(V)$, 
and $b(V)$ as the dimension of the smallest rational space 
containing $\Rel(V)$. Then $0\leq a(V) \leq n-d \leq b(V) \leq n$.

The configuration is called {\em rational} when $\Rel(V)$ is rational or, equivalently,         
$a(V)= n-d$ or $b(V)= n-d$. Otherwise $2+ a(V)\leq b(V)$, and all such values are possible. 
 
\subsubsection{Triangulations}\label{Triangulations}
Our main reference for triangulations and related concepts is the book \cite{DL-R-S}.
Let $\tau\subset\{1,\dots,n\}$. The {\em cone over} $\tau$ is defined as 
$\text{cone}(\tau)=\{\sum_{j\in\tau}\ \R_{\scriptscriptstyle{\geq 0}}\, v_{j} \}$.
By convention, $\text{cone}(\varnothing)=\{0_{E}\}$.
We say that $\tau$ 
is a {\em simplex} when the vectors indexed by $\tau$ are linearly independent 
(in particular, pairwise distinct).
A {\em simplicial cone} is a cone over a simplex. 

A {\em triangulation} $\bT$ of a configuration $V$ is a collection of simplices such that:
\begin{itemize}
\item If $\tau\in\bT$ and $\tau'\subset \tau$ then $\tau'\in\bT$;
\item For all $\tau,\tau' \in\bT$, $\text{cone}(\tau) \cap \text{cone}(\tau')=
 \text{cone}(\tau \cap \tau')$; 
\item $\cup_{\tau\in\bT}\ \text{cone}(\tau) \supset \text{cone}(V)$.
\end{itemize}
This definition allows that some vectors among $v_1,\dots, v_n$ 
do not belong to any simplex of $\bT$. We denote by $k\geq 0$ the number of such ``ghost vectors". 
We will always assume that they are at the end of the list $v_1,\dots, v_n$.
The pair $(V,\bT)$ is said to be a {\em triangulated configuration}. 

\subsubsection{Relations to other convex-geometric data}\label{convex geo data}
Suppose first that a triangulated configuration $(V,\bT)$ is given.

{\em Where is the fan?} The collection of cones on all of the simplices of 
$\bT$ is a {\em simplicial fan} $\Delta$, of dimension $d$. 
That is, a collection of simplicial cones such that: 
each nonempty face of a cone in $\Delta$ is a cone in $\Delta$; 
the intersection of any two cones in $\Delta$ is a face of each \cite{ziegler}. 
Notice that the fan $\Delta$ does not keep track of the ghost vectors and of the position of the other vectors on their respective rays. 
\rev{The non-ghost vectors play the role of generators of the rays of $\Delta$, as in \cite{PU}; they correspond to the vertices of the star-shaped simplicial sphere considered in \cite{Tam}.}

{\em Where is the polytope?} In general there is no relevant polytope associated to $(V,\bT)$. 
In the important special case of $\Delta$ being polytopal, there are infinitely many polytopes whose normal
fan is $\Delta$, all of the same combinatorial type. Some extra data is needed in order to 
determine a particular polytope.

{\em Where is the (generalized) lattice?} The $\Z$-submodule of $E$ 
generated by {\em all} the vectors 
$v_1,\dots,v_n$ is a quasilattice in $E$.
By {\em lattice} in $E$ we shall mean a 
quasilattice that is closed, equivalently of rank $d$.
The configuration $V$ being rational is equivalent to $Q$ being a lattice. 
It is well-known that fixing a rational fan but varying the lattice will change the associated toric variety. 
Analogously, starting from a triangulated configuration and modifying the quasilattice (by adding or deleting ghost 
vectors) will alter the geometry of the leaf space. This is exemplified in \ref{Nonrational CP1} which can be understood as a realization theorem, showing a substantial freedom in the construction even in dimension one.
\medskip

Conversely, assume given Prato's data of: a nonnecessarily rational simple polytope $P$ with $h$ facets; 
normal vectors $v_{1},\dots,v_{h}$; a quasilattice $Q$ containing these vectors. 
Choose $v_{h+1}, \dots,v_{n}$ such that $v_1,\dots,v_n$ generate $Q$. 
The vectors $v_{1},\dots,v_{h}$ generate the rays of 
the normal fan $\Delta$ of $P$. This fan determines a triangulation $\bT$
on $V=(v_{1},\dots,v_{n})$ with $v_{h+1}, \dots,v_{n}$ as ghost vectors. 

Actually some information is lost ---$P$ can't be recovered from $\Delta$---, 
but this information is not necessary to build the toric quasifold $X$ as a complex quotient \cite{cx}. 
As with toric varieties, the benefits of the symplectic reduction construction are an a priori symplectic/K\"ahler 
structure and compactness, whereas the advantages of the complex quotient are: an a priori complex structure; 
a generalization to the non polytopal case.
We will give more details later on how to encode and use that 
extra piece of information, that can exist only in the polytopal case.

Finally, starting from a stacky fan, we encode it in a similar way: we add ghost vectors to generate the 
ambient lattice, as in \cite{MV}.

\subsection{Construction of the LVMB-manifold $N$}\label{construction of N}
\subsubsection{Balanced and odd triangulations.}
\mbox{} Let $(V,\bT)$ be a triangulated vector configuration satisfying:
\begin{enumerate}
\item[({\it i})] $n-d=2m+1$ with $m$ a positive integer, 
\item[({\it ii})] $\sum v_i=0$.
\end{enumerate}
\rev{By ({\it ii}) and our assumption that $V$ spans the ambient space $E$, \new{the vectors of $V$ can not be contained in any half-space, so} cone$(V)=E$. Thus, the third defining property of triangulations (\cf \ref{Triangulations}) implies that} $\Delta$ is a complete fan. Conditions \rev{({\it i}) and ({\it ii})} are mild restrictions\rev{: starting with the weaker assumption that $(V,\bT)$ is a 
triangulated configuration whose associated fan is complete, we easily obtain ({\it i}) and ({\it ii})} while keeping both the 
quasilattice and the fan unchanged 
(this fact is used in [MV]). Namely, we apply the following algorithm:\\ 
Step 1. If  $\sum v_i \neq 0$, append $-\sum v_i$ as a new ghost vector of the configuration (and increase $n$ by $1$);\\
Step 2. If $n-d$ is even, append $0$ as a new ghost vector of the configuration (and increase $n$ by $1$);\\
Step 3. If $n-d=1$, append $0$ and $0$ as new ghost vectors of the configuration (and increase $n$ by $2$). 

\subsubsection{Virtual chamber and $U(\bT)$}\label{subsetsu} 
Denote the set of maximal simplices of $\bT$ by ${\{ \E_\alpha\}}_\alpha$.
Define the {\em virtual chamber} 
$\bE:={\big\{ \E_\alpha^{c} = \{1,\dots, n \} \setminus \E_\alpha \big\}}_\alpha$. By definition, virtual chambers correspond bijectively to triangulations of $V$ (\cf \cite{AS}).
For each $\alpha$, define $U_\alpha
:=\setofst{[z_1:\dots:z_n]\in \C\P^{n-1}}{ \forall j\in \E_\alpha^{c}, z_j\neq 0}$. 
Define $U(\bT):=\bigcup_\alpha U_\alpha$. 

\subsubsection{The dual configuration.}\label{The dual configuration}
Define a matrix $M\in \R^{n\times (2m+1)}$
by  
$$M=\begin{bmatrix}
1&a^1_1 & \dots & a^{2m}_1\\
  &           & \vdots \\
1&a^1_{n} & \dots & a^{2m}_{n}
\end{bmatrix},$$
where the columns form a basis of $\Rel(V)$.
Now define a vector configuration
$\hat{\Lambda}^\R=\big(\hat{\Lambda}^\R_1, \dots,\hat{\Lambda}^\R_{n}\big)$
 in $\R^{2m+1}$, called a {\em Gale dual} of $V=(v_1, \dots, v_{n})$, 
 and a configuration $\Lambda^\R=\big({\Lambda}^\R_1, \dots,{\Lambda}^\R_{n}\big)$
 in $\R^{2m}$ by  
$$
M=\begin{bmatrix}
\horiz\ \hat{\Lambda}^\R_1\ \horiz \\
      \vdots\ \  \\
\horiz\ \hat{\Lambda}^\R_{n}\ \horiz \\
\end{bmatrix}
=\begin{bmatrix}
1&\!\!\! \horiz \Lambda^\R_1 \horiz \\
  &      \vdots\ \ \ \  \\
1&\!\!\! \horiz \Lambda^\R_{n} \horiz \\
\end{bmatrix}.$$

Notice that $M$ is only defined up to right multiplication 
by a matrix of form $T=\mysmallmatrix{1&B \\0&A}$
where $B=(b_1,\ldots,b_{2m})\in\R^{2m}$ and $A\in GL(2m,\R)$. 
Therefore, a Gale dual is not unique, 
and $\Lambda^{\R}$ is only defined 
up to the invertible real affine transformation of the ambient  $\R^{2m}$ given by 
$X\mapsto XA+B$. Thus, $\Lambda^{\R}$ is to be seen as a configuration 
of {\em points}, i.e., affine objects. We refer to Sect. \ref{examples} 
for examples.

\rev{The quantitative measures of nonrationality of the dual configurations 
$V$ and $\hat{\Lambda}^\R$ are linked by the relations
$$a(V)+b(\hat{\Lambda}^\R)= n \quad\text{ and }\quad a(\hat{\Lambda}^\R)+ b(V)=n,$$ 
which follow from $\Rel(\hat{\Lambda}^\R)=\textrm{Ker }M^t=(\textrm{Im } M)^{\perp}= \Rel(V)^{\perp}$. 
We note also that $a\big(\hat{\Lambda}^\R\big)$
is denoted $a$ in \cite[Th. 4]{M}, where it is shown that the algebraic dimension of $N$ is at least $a$, with equality in the absence of ghost vectors.
}

\subsubsection{The $\C^{m}$-action and $N$.}\label{CmandN}\mbox{}
Consider the holomorphic $\C^m$-action on $U(\bT)$ defined by
\begin{equation}\label{action giving N}
\applicnn{\C^m\times U(\bT)}{\ U(\bT)}
{\Big( \vu\ ;\ [z_1:\dots:z_{n}] \Big)}
{\ [e^{2\pi i \Lambda_1(\vu)}z_1:\dots : e^{2\pi i \Lambda_{n}(\vu)} z_{n}],}
\end{equation}
where 
$$\Lambda_j := \bvect{a_j^1+i a_j^{m+1}\\ \vdots \\a_j^m + i a_j^{2m}} \in \C^{m} 
$$ 
with $a_j^1, \dots, a_j^{2m}$ denoting the entries of $\Lambda_j^\R$\new{, and 
$\Lambda_{j}(\vu)$ denotes the dot product.}

Bosio has given in \cite{Bos} sufficient conditions for this action to be proper and cocompact. 
We show below that action (\ref{action giving N}) is free and Bosio's conditions hold, thus the quotient of 
$U(\bT)$ by this action is a compact complex manifold that we denote $N$. 
\rev{Note that acting on $V$ by a linear automophism of $E$ is immaterial for the construction we have described, since $\Rel(V)$ is unchanged by such a transformation.}
\rev{We refer to \cite{Bos, M} for properties of $N$, and note here that the standard holomorphic $(\C^*)^n$-action on $\C P^{n-1}$ induces a {\em decomposition} of $N$ (\cf \cite[p.~36]{gmp}): 
define $N(\tau)\subset N$ as the image in $N$ of 
$\setofst{[\vz] \in U(\bT) }{\ z_j \neq0\;\text{iff}\; j\not\in\tau}$.
Then $N$ is the disjoint union of the $(\C^*)^{n}$-orbits 
$\coprod_{\tau\in\bT}N(\tau)$, with a unique open orbit $N(\varnothing)$. 
}

\subsubsection{Proof that Bosio's conditions hold}  \label{Proof that Bosio's conditions hold}
By properties of Gale duality (see \cite{DL-R-S} Def. 5.4.3 and the comment below), for each 
$\alpha, \setofst{\hat{\Lambda}^\R_j}{j\in\E_\alpha^c}$ is 
a simplex, i.e., a linear basis of 
$\R^{2m+1}$. 
Let $P_\alpha$ denote 
the convex hull of 
$\setofst{\Lambda^{\R}_j}{j\in\E_\alpha^c}\subset\R^{2m}$.
Then $\mathring{P}_\alpha\neq \varnothing$, and 
it follows that action (\ref{action giving N}) 
has trivial isotropy at any element of $U_{\alpha}$, so this action is free \new{on $U(\bT)=\cup_{\alpha} U_{\alpha}$.
The statement on the isotropy, found in \cite[Rem~1.1]{Bos} or Meersseman's thesis, 
is proved as follows: let $\vu\in \C^m$ be in the isotropy at $\vz\in U_{\alpha}$ and suppose without loss of generality 
that $n\in \E_\alpha^c$. 
This implies
$\text{Im}\ [(\Lambda_j-\Lambda_n)(\vu)]=0$ for all 
$j\in \E_\alpha^c\setminus\{n\}$. This in turn implies 
$(\Lambda^{\R}_j-\Lambda^{\R}_n)\Big(\begin{array}{c}\text{Im}\ \vu\\\text{Re}\ \vu\end{array}\Big)=0$. 
Since  $\mathring{P}_\alpha\neq \varnothing$, 
the vectors $\Lambda^{\R}_j-\Lambda^{\R}_n$, with $j\in \E_\alpha^c\setminus\{n\}$, 
are a basis of $\R^{2m}$, therefore
$\vu=0$.
}
The result below belongs to a circle of ideas that appear 
in the works of Bia\l ynicki-Birula and \'Swi\c ecicka. 
Similar results include also \cite{BH} Lemma 3.5 
and \cite{Tam} Prop. 2.3 and Cor. 2.4.

\begin{prop}
Bosio's conditions hold here, i.e.,
\begin{enumerate}
\item[(i)] $\mathring{P}_\alpha\cap 
  \mathring{P}_\beta  \neq\varnothing$ for every $\alpha, \beta$; 
\item[(ii)] for every $\mathcal E_{\alpha}^{c}\in \bE$ and every 
$i \in \E_\alpha$,\\ 
there exists $k\in \mathcal E_{\alpha}^{c}$ such that 
$\big(\mathcal E_{\alpha}^{c}  \setminus \{k\}\big)\cup \{i\}\in \bE$.
\end{enumerate}
\end{prop}
\begin{proof}
(i) Pick in $\bT$ any two distinct maximal simplices $\E_\alpha$ and $\E_\beta$, 
and choose a linear form $\varphi$ that separates the respective cones,
in the sense that $\varphi$ is positive on 
$\text{cone}(\E_\alpha)$ and negative on $\text{cone}(\E_\beta)$, 
except on $\text{cone}(\E_\alpha)\cap \text{cone}(\E_\beta)$, where it is zero.
A linear evaluation such as 
$$\Big(\varphi(v_1),\dots, \varphi(v_{n}) \Big)$$
corresponds (\cf \cite{DL-R-S} p. 244) to a linear relation on the Gale dual 
with coefficients given by $\varphi(v_1),\dots, \varphi(v_{n})$. 
Here the relation has the form
$$
\sum_{j \in \E_\alpha \setminus \E_\beta} a_j \hat{\Lambda}^\R_j
- \sum_{j \in \E_\beta \setminus \E_\alpha} b_j \hat{\Lambda}^\R_j
+ \sum_{j \not\in \E_\alpha \cup \E_\beta} c_j \hat{\Lambda}^\R_j
=0\text{,}
$$
where all $a_j$'s and $b_j$'s are positive. 
For all  $j \not\in \E_\alpha \cup \E_\beta$, we write 
$c_j$ as the difference of two positive numbers $a_j - b_j$. 
Then 
$$
\sum_{j \in \E_\alpha \setminus \E_\beta} a_j \hat{\Lambda}^\R_j
+ \sum_{j \not\in \E_\alpha \cup \E_\beta} a_j \hat{\Lambda}^\R_j
=
\sum_{j \in \E_\beta \setminus \E_\alpha} b_j \hat{\Lambda}^\R_j
+ \sum_{j \not\in \E_\alpha \cup \E_\beta} b_j \hat{\Lambda}^\R_j\;
\text{, i.e.,}
$$
$$\sum_{j \in \E_\beta^c} a_j \hat{\Lambda}^\R_j
=
\sum_{j \in \E_\alpha^c} b_j \hat{\Lambda}^\R_j. 
$$
Thus ${\displaystyle \sum_{j \in \E_\beta^c} a_j =  \sum_{j \in \E_\alpha^c} b_j=: s}$, and 
$$\frac{1}{s}\sum_{j \in \E_\beta^c} a_j {\Lambda}^\R_j
=
\frac{1}{s}\sum_{j \in \E_\alpha^c} b_j {\Lambda}^\R_j.$$
The left hand side and right hand side belong to 
$\mathring{P}_\beta$ and 
$\mathring{P}_\alpha$ respectively. 
Therefore the intersection 
is nonempty.
 
 \medskip
(ii) Pick $\E_\alpha^c\in \bE$ and  
$i\in \E_\alpha$. The facet of $\text{cone}(\E_\alpha)$ 
determined by  omitting $v_i$ 
is shared by one and only one maximal cone, say $\text{cone}(\E_\beta)$. 
Then
$\E_\beta= \big(\E_\alpha \setminus \{i\}\big) \cup \{k\}$ 
for some $k$, and $k\not\in \E_\alpha$ by convexity 
of $\text{cone}(\E_\alpha)$. Then 
$\big(\E_\alpha^c \setminus \{k\}\big)\cup \{i\} =\E_\beta^c \in \bE$.
\end{proof}

\subsection{The foliation $\calF$ on $N$}\label{foliation}
Consider on $U(\bT)$ the following holomorphic action by $\C^{2m}$:
\begin{equation}\label{action giving X}
\vt . [z_1:\dots:z_n] = 
[\expo{\lr{1}(\vt)} z_1:\dots: \expo{\lr{n}(\vt)} z_n].
\end{equation}
Fix a $[\vz] \in U(\bT)$. Direct computations show that the isotropy 
at $[\vz]$ is a closed \Z-module 
$L_{\vz}\subset \R^{2m}\subset\C^{2m}$
of rank at most $2m$.

\smallskip
Action (\ref{action giving X}) commutes with (\ref{action giving N}), so it descends to $N$.
The restriction of action (\ref{action giving X}) to 
$$
\C^{m}_{N}:=\{\vt\in\C^{2m}\;|\;
\vt=\left(\begin{array}{c}\vu\\i\vu\end{array}\right),\,\vu\in\C^m\}$$
gives action (\ref{action giving N}). 
Define 
$$\C^{m}_{\calF}:=\{\vt\in\C^{2m}\;|\;
\vt=\left(\begin{array}{c}\vv\\ 0 \end{array}\right),\,\vv\in\C^m\}.$$
The projection
$\pi: \C^{2m} = \C^m_{N} \oplus \C^m_{\calF} \rightarrow \C^m_{\calF}$ is given
by $(\vx,\vy)\mapsto (\vx+i\vy, 0)$.
The isotropy of $[\vz]\in N$ for the action of $\C^{m}_{\calF}$ on $N$ 
is $\pi(L_{\vz})$.
Therefore this action has discrete isotropy, so 
it induces on $N$ a smooth foliation $\calF$ of dimension $m$. 
In the polytopal case, this foliation appears in \cite{LN} and \cite{M} 
(cases $m=1$ and $m\geq 1$ respectively). The foliation $\calF$ is holomorphic, 
and in particular transversely orientable. We show below the stronger 
statement that $\calF$ is homologically orientable (\cf (\ref{hom orient}) in Sect. \ref{Stanley}).

The leaf $\calF_{\vz}$ through a point $[\vz]\in N$ is the image, via an injective immersion, of 
$\C^m_{\calF}/ \pi(L_{\vz})$.
\rev{By varying the choice of the Gale dual, 
the \Z-module $L_{\vz}$ becomes $A^{-1}L_{\vz}$, 
with $A\in GL(2m,\R)$ (\cf Sect.~\ref{The dual configuration}), so 
the holomorphic structure on 
$\calF_{\vz}$ varies among all complex abelian groups on a fixed topological type. There is a unique $\tau$ such that $\calF_{\vz} \subset N(\tau)$. Define a  
subconfiguration of $\hat{\Lambda}^\R$ by 
$\hat{\Lambda}^\R(\tau):=(\hat{\Lambda}^\R_j)_{j\not\in\tau}$.  
By computing $\textrm{rank}\big(\pi(L_{\vz})\big)$ we obtain
the topological type of the leaf
$$\calF_{\vz}\approx {(S^{1})}^{B(\tau)-1}\times \R^{2m-B(\tau)+1}$$
where
$B(\tau)=n-\#\tau-b\big(\hat{\Lambda}^\R(\tau)\big)$. 
The topological type of the leaf closure is 
$$\overline{\calF_{\vz}}\approx(S^1)^{A(\tau)-1}$$ 
where $A(\tau)=n-\#\tau-a\big(\hat{\Lambda}^\R(\tau)\big)$.
In particular, these topological types depend on $V$ and $\tau$, 
but not on the choice of the Gale dual. }

\new{Generic leaves (i.e. lying in the open orbit $N(\varnothing)$) 
correspond to $\tau=\varnothing$. 
Since $a(V)=n-b(\hat{\Lambda}^\R)$ (\cf \ref{The dual configuration}), 
they are homeomorphic to 
${(S^{1})}^{a(V)-1}\times \R^{2m-a(V)+1}$.
If the configuration is rational, that is $a(V)=b(V)=2m+1$, all leaves are closed 
(\cf \cite{MV}).
On the other hand there are nonrational configurations $V$ such that $a(V)=1$; in these cases the generic 
leaf is $\C^m$.}
\subsubsection{The leaf space}\label{The leaf space}
Let $\Delta$, $v_1,\ldots,v_h$, and $Q$ 
be the fan, the rays generators (i.e., non-ghost vectors, so $h=n-k$), and the quasilattice associated to $(V,\bT)$ (see Sect.~\ref{convex geo data}).
 \rev{From the Audin-Cox construction and its nonrational complex generalization \cite{cx},} it is known that to this data there corresponds a geometric quotient $X=\rev{U'}(\Delta)/G$, 
where $\rev{U'}(\Delta)$ 
is an open subset of $\C^h$ that 
depends on the combinatorics of $\Delta$, and $G$ is a complex subgroup of $(\C^*)^h$ that depends on $Q$ and on the 
vectors $v_1,\ldots,v_h$.
If the configuration is rational (resp. nonrational), then $X$ is 
a complex manifold or a complex orbifold (resp. a non Hausdorff complex quasifold) of dimension $d$,
acted on holomorphically by the torus (resp. {\em quasitorus}) $\C^d/Q$ (\cf \cite{A,cox,p,cx}; 
the construction in \cite[Thm~2.2]{cx} can be adapted to the nonpolytopal case).
\rev{Quasifolds generalize orbifolds: the local model is a quotient of a manifold by the smooth action of a finite or countable group, non free on a closed subset of topological codimension at least $2$ \cite{p}.}
Let $(N,\calF)$ be any foliated complex manifold corresponding to $(V,\bT)$. 
The complex structure induced by $(N,\calF)$ on the leaf space 
depends on the initial data \rev{ $(V,\bT)$, but not on the choice of a Gale dual}. 
\begin{rem}\label{downstairs foliation}{\rm 
The action of the group $G$ does induce a holomorphic foliation on $U\rev{'}(\Delta)$.
However, since $G$ is in general, for rational and nonrational configurations, not connected \rev{(\cf Ex.~\ref{arbitrary p q})},
the leaf space is {\em not} $X$. This problem is overcome in our construction by ``increasing the dimension".}
 \end{rem}
\subsubsection{The foliation is Riemannian}\label{F is Riemannian}
Consider the $(S^1)^{n-1}$-action on $N$ induced by the $(\C^*)^{n-1}$-action on $\C P^{n-1}$, 
and construct a Riemannian metric on the compact 
manifold $N$ such that the compact group $(S^1)^{n-1}$ acts by isometries. Now we observe that  
$\C_{\calF}^{m}$ acts on $N$ as a subgroup of $(S^1)^{n-1}$: fix $[\vz]\in N$ and 
$\vt_1=\left(\begin{array}{c}\vv\\ 0 \end{array}\right)\in\C^m_{\calF}$.  Define
$\vt_2=\left(\begin{array}{c}\vu\\i\vu\end{array}\right)\in\C^m_{N}$, with $u=-i\text{Im}(\vv)$.
Then $\vt_1\cdot(\vt_2\cdot[\vz])=\vt\cdot[\vz]$ with 
$\vt=\left(\begin{array}{c}\text{Re}(\vv)\\\text{Im}(\vv)\end{array}\right)$, where the action used here is action
(\ref{action giving X}).
Now, the $\C^m_{\calF}$-action being locally free implies that the induced foliation $\calF$ is Riemannian
\cite[Ex.~2, p.100]{molino}. The same argument shows that the foliation is moreover Killing 
\cite{moz}.

\section{Topological results on the basic cohomology algebra}
\label{Betti numbers}

\noindent
In this section we show how the combinatorics of a balanced and odd triangulated configuration $(V,\bT)$
relate to the basic Betti numbers of any foliated manifold $(N,\calF)$ built from
$(V,\bT)$. The formulas are the same as the usual Betti numbers of simplicial toric varieties.

For the combinatorial part we refer the reader to 
\cite[Sect. 8.3]{ziegler}; for basic cohomology, see \cite{Ton}. We recall definitions and results in the form we need for our purposes.

\subsection{Shellings and $h$-vector}\label{shellings and h}
Fix a triangulated vector configuration $(V,\bT)$. In particular, $\bT$ 
is an abstract simplicial complex of pure dimension $d-1$ (topologically, a sphere). The  {\em dimension} of a (possibly empty) simplex $\tau\in\bT$ is $\#\tau-1$. 
Recall that the {\em $f$-vector} 
$(f_{-1},f_0,f_1,\ldots,f_{d-1})$
records the number of simplices in each dimension.
The fan $\Delta$ gives a ``linear realization'' of this simplicial complex, with 
simplices of dimension $l$ corresponding bijectively to cones of dimension $l+1$.

A {\em shelling} of $\bT$ (or of $\Delta$) is a linear ordering of the maximal simplices 
$\E_{1},\dots, \E_{f_{d-1}}$ such that for all $\alpha\geq 2$, $\text{cone}(\E_{\alpha})$ intersects 
$\text{cone}(\E_{1})\cup\dots\cup \text{cone}(\E_{\alpha-1})$ along a nonempty union of facets of 
$\text{cone}(\E_{\alpha})$. The number of such facets, called the {\em index} of $\E_{\alpha}$ w.r.t. the shelling, is denoted $i_{\alpha}$. We take $i_{1}=0$. 

Polytopal fans are shellable, i.e., they admit a shelling 
(\cite[Sect. 8.2]{ziegler}). 
The {\em $h$-vector} $(h_{0}, h_{1},\dots, h_{d})$
of $\bT$ (or $\Delta$) records 
the number of maximal simplices of each index in a given shelling. 
It is well-known that the $h$-vector 
is completely determined by the $f$-vector ---in particular, 
it is independent of the choice of a shelling--- and conversely it determines the $f$-vector. 

\subsection{Basic cohomology}
Let $M$ be a smooth manifold with a smooth foliation $\calG$. 
A differential form $\omega\in \Omega^{\bullet}(M)$ is called 
{\em basic} when for all vector fields $X$ tangent to $\calG$, 
$\iota_{X} \omega =0$ and $\iota_{X} d\omega =0$. 
When the foliation is given by the orbits of a Lie group $G$,
this means that the form is $G$-invariant and its kernel contains 
the tangent space to $\calG$. The cohomology 
of the complex of basic forms is in some sense the de Rham cohomology of the leaf space. The dimensions of these groups are called the {\em basic Betti numbers}. 
An example that gives some intuition for the proofs 
below is the torus $M=S^{1}\times S^{1}$ with $\calG$ given by lines of 
slope $s$. When $s$ is rational, the leaf space is a circle and 
$b_{\calG}^{1}(M)=1$.
When $s$ is irrational, the leaf space is not Hausdorff but, again, $b_{\calG}^{1}(M)=1$. Cohomologically, the leaf space is still a circle. 
Notice however that the basic Betti numbers of a foliated compact manifold can be infinite-dimensional in general, and that basic Betti numbers are not invariant under small deformations of Riemannian foliations \cite[Example 7.4]{No}.

\subsection{Computation of the basic Betti numbers}
\begin{thm}\label{prop betti numbers}
Let $(V,\bT)$ be a shellable, balanced and odd triangulated vector configuration, with $\text{dim}(V)=d$ and $h$-vector 
$(h_{0},\dots,h_{d})$.
Let $(N,\calF)$ be any foliated manifold built from $(V,\bT)$.
Then the basic Betti numbers are
$$b_{\calF}^{2j+1}(N)=0$$
and
$$b_{\calF}^{2j}(N)=h_{j}$$
for $j=0,\ldots,d$.
\end{thm}
\begin{proof}
We use a ``Morse-theoretic'' method due to Khovanskii for simple polytopes. Working dually with simplicial fans, we see that his method extends (from polytopal fans) to shellable fans.
Let $\E_1,\ldots,\E_{f_{d-1}}$ be a shelling of\, $\bT$. 
Consider the open subsets $U_\alpha$ defined in \ref{subsetsu} and their image, $N_\alpha$, in $N$.
We consider the $\calF$-saturated open covering of $N$ defined as follows:
$$\begin{array}{l}
W_1={N}_1\text{,}\\
W_\alpha=W_{\alpha-1}\cup {N}_\alpha\text{,}\quad \alpha=2,\ldots,f_{d-1}.
\end{array}
$$
Therefore
$${N}_1=W_1\subset W_2\subset \cdots \subset W_{f_{d-1}}=N.$$

We compute inductively the basic cohomology of the foliated manifolds $W_\alpha$ by means of a Mayer-Vietoris sequence. 
For this we need a {\em basic partition of unity}: pick any partition of unity relative to
the decomposition $W_\alpha=W_{\alpha-1}\cup N_\alpha$. 
The natural $(S^1)^d$-action on $\C^d$ descends to $N$. Averaging out 
the functions over this action will in particular make them constant on the leaves. 
 
From Lemma ~\ref{lemma-sphere} below, case $r=0$, we 
know that $b_{\calF}^{l}(W_{1}={N}_1)=\delta_{0,l}$. 
Now fix $\alpha\geq 2$ and make the induction hypothesis: 
$$(\mathcal H_{\alpha-1}) \qquad \text{if $l$ is odd then}\ b_{\calF}^{l}(W_{\alpha-1})=0.$$
We claim that $$b_{\calF}^{l}(W_{\alpha})=\begin{cases} 
b_{\calF}^{l}(W_{\alpha-1}) & \text{ if } l \neq 2i_{\alpha}, \\
b_{\calF}^{l}(W_{\alpha-1})+1 & \text{ if } l = 2i_{\alpha}.
\end{cases}$$ 
In particular, $(\mathcal H_{\alpha})$ holds, and the theorem follows by induction. 

Now we prove the claim. 
Using the notation of Lemma ~\ref{lemma-sphere}, 
remark first that $W_{\alpha-1}\cap N_\alpha$ is of the form $N_{\alpha,\tau}$, 
where 
$\tau$ is the {\em restriction} of $\E_{\alpha}$, defined in  
\cite[8.3]{ziegler} as 
\begin{equation}\label{restriction}
\tau=\setofst{i\in\E_{\alpha}}{(\E_{\alpha}\setminus i)  \subset \E_{\beta} \text{ with } \beta<\alpha}. 
\end{equation}
Notice that $\#\tau=i_{\alpha}$. 

Then Lemma ~\ref{lemma-sphere}
tells us that $W_{\alpha-1}\cap N_\alpha$ has no basic cohomology in positive even dimension.  
Thus by Mayer-Vietoris, for any odd integer $p$,\\ 
$0 \rightarrow H_{\calF}^{p}(W_{\alpha}) \rightarrow 
\underbrace{H_{\calF}^{p}(W_{\alpha-1})}_{\scriptscriptstyle =0 \text{ by } (\mathcal H_{\alpha-1})} \oplus 
\underbrace{H_{\calF}^{p}(N_\alpha)}_{\scriptscriptstyle =0 \text{ by Lemma ~\ref{lemma-sphere}}}
\rightarrow 
\underbrace{H_{\calF}^{p}(W_{\alpha-1}\cap N_\alpha)}_{(*)}\rightarrow H_{\calF}^{p+1}(W_\alpha)\rightarrow H_{\calF}^{p+1}(W_{\alpha-1}) \oplus 
\underbrace{H_{\calF}^{p+1}(N_\alpha)}_{\scriptscriptstyle =0 \text{ by Lemma ~\ref{lemma-sphere}}}
\rightarrow 0$

We see that the second term, $H_{\calF}^{p}(W_{\alpha})$, must be zero. 
Again by Lemma~\ref{lemma-sphere}, 
$(*)$ is zero unless 
$p=2i_{\alpha}-1$, in which case it is of dimension one. 
\end{proof}
\begin{lem}\label{lemma-sphere} 
Let $(V,\bT)$ be a balanced and odd triangulated vector configuration in a vector space of dimension $d$.
Let $\tau$ be a subset of a maximal simplex $\E_{\alpha}$ of\, $\bT$. 
We define an $\calF$-saturated open subset of $N$, denoted  
$N_{\alpha,\tau}$, as the image in $N$ of 
$$U_{\alpha,\tau}=U_{\alpha} \setminus 
\setofst{[z_1:\dots:z_n] }{\  \forall j\in \tau, z_j =  0}.$$   
In particular, $N_{\alpha,\tau}= N_{\alpha}$ when $\tau$ is empty. 
Then, denoting $r=\#\tau$, 
$$\forall l\geq 0,\ 
H_\calF^l(N_{\alpha,\tau})\approx
\begin{cases}
\R & \text{if}\quad l=0\quad \text{or}\quad l=2r-1,\\
{0} & \text{otherwise.}
\end{cases}
$$
In other words, the leaf space $N_{\alpha,\tau}/\calF$ is cohomologically a point when $r=0$,
and a $2r-1$-sphere when $r>0$. 
\end{lem}
\begin{proof} 
By definition of $\calF$ we have $H_\calF^l(N_{\alpha,\tau})\approx
H^{l}\Big(\Omega^\bullet_{\C^{2m}}\big(U_{\alpha,\tau}\big)\Big)$, where 
$\Omega^\bullet_{\C^{2m}}\big(U_{\alpha,\tau}\big)$ 
denotes the complex of forms on $U_{\alpha,\tau}$ that are basic 
with respect to the foliation induced by the $\C^{2m}$-action (\ref{action giving X}). 

Suppose for now that $r>0$, and assume for simplicity that 
$\E_\alpha= \{1,\dots,d\}$ and 
$\tau= \{1,\dots,r\}$. 
We know that 
$\left\{ \lr{j}-\lr{n} \right\}_{ j=d+1 \dots n-1}$ is an \R-basis of $\R^{2m}$, 
so it is a \C-basis of $\C^{2m}$. This implies surjectivity of 
the map
$$\applic{g}{\big (\C^r\! \smallsetminus\! \{0\} \times \C^{d-r} \big) \times \C^{2m}}{\ U_{\alpha,\tau}}
{\big( (z_1,\dots,z_d) ; \vt \big)}
{\ \vt . [z_1:\dots:z_d:\underbrace{1:\dots:1}_{2m+1}].}$$
On the other hand,
$g\big( (z_1,\dots,z_d) ; \vt \big)=g\big( (w_1,\dots,w_d) ; \vs \big)$ 
is equivalent to
$$\begin{cases}
(w_1,\dots,w_d)=
\Big( \expo{\left( \lr{1}-\lr{n} \right)(\vt-\vs)} z_1,\dots, 
\expo{\left( \lr{d}-\lr{n} \right)(\vt-\vs)} z_d \Big)\\
\left( \lr{j}-\lr{n} \right)(\vt-\vs) \in \Z, \ j=d+1 \dots n-1.
\end{cases}$$

\medskip
The second condition implies that $\vt-\vs \in L_{\vu}$
with $[\vu]$ any point in $N(\E_\alpha)$. Therefore $L_{\vu}$ is a lattice in $\R^{2m}$ that we denote $\Gamma$.
This shows that the fibers of $g$ are the orbits of the $\Gamma$-action on 
$\big (\C^r\! \smallsetminus\! \{0\} \times \C^{d-r} \big) \times \C^{2m}$ defined by
\begin{equation} \label{gamma action}
\vgamma. \big( (z_1,\dots,z_d) ; \vt \big) = 
\Big( (\expo{\left( \lr{1}-\lr{n} \right)(\vgamma)}z_1,\dots, 
\expo{\left( \lr{d}-\lr{n} \right)(\vgamma)}z_d) ; \vt - \vgamma \Big).
\end{equation}
Notice that the action of $\Gamma$ on the first factor does not depend on the choice of a Gale dual: 
changing this choice, 
$\lr{j}-\lr{n}$ becomes $(\lr{j}-\lr{n})A$ and $\Gamma$ 
becomes $A^{-1}\Gamma$ (\cf \ref{The dual configuration}), 
thus $(\lr{j}-\lr{n})A(A^{-1}\gamma)=(\lr{j}-\lr{n})(\gamma)$.
Remark also that the map $g$ induces a foliation-preserving homeomorphism
$$\big (\C^r\! \smallsetminus\! \{0\} \times \C^{d-r} \big) \times \C^{2m}/\simeq U_{\alpha,\tau}.$$
This in turn implies
\begin{equation}\label{uniformizing nbhd}
\big (\C^r\! \smallsetminus\! \{0\} \times \C^{d-r} \big) \times \C^{m}_{\calF}/\Gamma\simeq
N_{\alpha,\tau}.
\end{equation}

Now, $\omega \mapsto g^* (\omega)$ maps isomorphically 
the complex $\Omega^\bullet_{\C^{2m}}\big(U_{\alpha,\tau}\big)$ 
onto the complex $\mathcal C^{\Gamma}$ of forms on $\big (\C^r\! \smallsetminus\! \{0\} \times \C^{d-r} \big) \times\C^{2m}$ that are: (a) 
$\Gamma$-invariant; (b) basic with respect to 
the foliation with leaves $\{\underline{z}\}\times \C^{2m}$. But a form 
satisfies condition (b) if and only if it is the pull-back of a form by the projection 
$\big (\C^r\! \smallsetminus\! \{0\} \times \C^{d-r} \big) \times\C^{2m} \rightarrow \C^r\! \smallsetminus\! \{0\} \times \C^{d-r}$. Therefore 
$\mathcal C^{\Gamma}$ is (isomorphic to) the complex of $\Gamma$-invariant forms
$\Omega^{\bullet}\Big(\C^r\! \smallsetminus\! \{0\} \times \C^{d-r}\Big)^{\Gamma}$.  
By (\ref{gamma action}), we see that the action of $\Gamma$ factors through the standard 
$(S^{1})^{d}$-action on $\C^{d}$. Therefore we can apply \cite[Lemma~2.2]{Bat} to conclude that for every $l\geq 0$,
\begin{align*}
H^{l}\Big(\Omega^{\bullet}\big(\C^r\! \smallsetminus\! \{0\} \times \C^{d-r}\big)^{\Gamma}\Big)
&\approx 
H^{l}\Big(\Omega^{\bullet}\big(\C^r\! \smallsetminus\! \{0\} \times \C^{d-r}\big)\Big)\\
&\approx H^{l}(\C^r\! \smallsetminus\! \{0\}) \approx H^{l}(S^{2r-1}).
\end{align*}
In the case $r=0$ the proof is similar: replace every 
$\C^r\! \smallsetminus\! \{0\} \times \C^{d-r}$ with $\C^{d}$, 
and omit the last line.
\end{proof}
After a first version this paper was completed, we came across the article 
\cite{GT} by O. Goertsches and D. T\"oben,
that contains results in the spirit of our Th.~\ref{prop betti numbers}. We use their techniques in Sect. \ref{gen degree two}. 

\subsection{The basic cohomology algebra is generated in degree two}\label{gen degree two}

Let $(V,\bT)$ be a balanced and odd triangulated vector configuration, with 
associated fan denoted $\Delta$. 

In this section we assume {\em polytopality}, i.e. there exists 
a (necessarily) simple polytope $P\subset \R^{d}$ whose normal fan is $\Delta$. 
There is an inclusion-reversing duality between faces of 
$P$ and cones of $\Delta$ (or simplices of $\mathcal T$). In particular, each vertex of $P$ corresponds to a 
maximal simplex of $\mathcal T$. We fix a shelling of $\mathcal T$ in the following way: thinking of the last coordinate of the ambient space as the ``height'', we rotate $P$ until no two of its vertices have same height. 
We order the vertices from lowest to highest. It is easy to check that the corresponding order on 
the maximal simplices is a shelling of $\mathcal T$, that we denote $\E_{1},\dots, \E_{f_{d-1}}$. 

Fix a maximal simplex $\Ea$. Denote its restriction, defined in (\ref{restriction}),
with respect to the shelling by $\tau_{\alpha}$, and let 
$\tau_{\alpha}^{-}:=\Ea\smallsetminus\tau_{\alpha}$. Now denote by $V_{\alpha}$, $F_{\alpha}$ and $F_{\alpha}^{-}$ the closed 
faces of $P$ that are dual to 
$\Ea$, $\tau_{\alpha}$ and $\tau_{\alpha}^{-}$ respectively. We remark that $V_{\alpha}$ 
is the lowest (resp. highest) vertex of $F_{\alpha}$ (resp. $F_{\alpha}^{-}$).
Recall that for any simplex $\sigma\in\mathcal T$,
$N(\sigma)$ is the image in $N$ of the set
$\setofst{[\vz] \in U(\bT) }{\ z_l \neq0\;\text{iff}\; l\in\sigma^c}$, 
and
$\overline{N(\sigma)}$ is the disjoint union $\sqcup_{\sigma\subset\sigma'}N(\sigma')$. 

\begin{lem}\label{Lemma_intersections}\mbox{} 
\begin{enumerate}
\item[(\em i)] For all $\alpha, \beta$ such that $ \beta<\alpha$, $\overline{N(\tau_{\alpha})}$ 
and $\overline{N(\tau_{\beta}^-)}$ are disjoint;
\item[(\em ii)] 
For all $\alpha$, $\overline{N(\tau_{\alpha})}$ and $\overline{N(\tau_{\alpha}^-)}$ intersect transversally along 
$N(\Ea)$, which is a compact leaf of $\calF$.
\end{enumerate}
\end{lem}
\begin{proof}\ \\
\begin{enumerate}
\item[(\em i)] We know that $F_{\alpha}$ has no point below $V_{\alpha}$, and $F_{\beta}^{-}$ has no point above $V_{\beta}$. As $V_{\beta}$ is lower than $V_{\alpha}$,  the closed faces $F_{\alpha}$ and $F_{\beta}^{-}$ are disjoint, i.e. they have no common face. Dually, this means that no simplex contains both $\tau_{\alpha}$ and $\tau_{\alpha}^{-}$. This implies that the disjoint unions making up $\overline{N(\tau_{\alpha})}$ and  
$\overline{N(\tau_{\beta}^-)}$ have no common component, and the result follows.
\item[(\em ii)] 
On the image in $N$ of $U_{\alpha}$ (which is of the form $\C^{d}\times\C^m_{\calF}/\Gamma$; \cf (\ref{uniformizing nbhd})\,), the two submanifolds $\overline{N(\tau_{\alpha})}$ and $\overline{N(\tau_{\alpha}^-)}$ become respectively $\{0\}\times\C^{d-r}\times\C^m_{\calF}/\Gamma$ and $\C^r\times\{0\}\times\C^m_{\calF}/\Gamma$.
The intersection is $\{0\}\times\{0\}\times\C^m_{\calF}/\Gamma$, a compact torus.
\end{enumerate}
\end{proof}
From the polytopality assumption, we know that $N$ is an LVM-manifold (\cf \ref{preliminaries}). 
For each simplex $\sigma\in\mathcal T$, the subset $\overline{N(\sigma)}$ is defined by the vanishing of 
certain coordinates. From Property 5 in \ref{preliminaries}, it follows that $\overline{N(\sigma)}$ is 
itself an LVM-manifold. 
In particular it is a smooth, compact, $\calF$-saturated complex submanifold. 
Moreover, $\calF$ is holomorphic, so 
there is a fiber orientation form on the normal bundle $\scalebox{1.5}{$\nu$}\overline{N(\sigma)}=TN_{\vert\overline{N(\sigma)}}/T\overline{N(\sigma)}$ that is invariant by the foliation's holonomy.
It follows that this normal bundle is oriented as a foliated bundle (\cf paragraph above \cite[Cor.~4.8]{toben}, where the author defines  on the normal bundle a foliation \new{$\mathcal{G}$} induced by the holonomy of $\calF$). 
\rev{In this situation, it is possible \cite[Sect. 4]{toben} to define the integration along the fibers of basic 
forms with compact vertical support
$\pi_*\colon \Omega_{\mathcal{G},cv}^{\bullet\,+\,\#\sigma}(\scalebox{1.5}{$\nu$}\overline{N(\sigma)})\longrightarrow \Omega^\bullet_\calF(\overline{N(\sigma)})$. 
This yields a homomorphism
in cohomology 
$\pi_*\colon H_{\mathcal{G},cv}^{\bullet\,+\,\#\sigma}(\scalebox{1.5}{$\nu$}\overline{N(\sigma)})\longrightarrow H^\bullet_\calF(\overline{N(\sigma)})$, 
which is
an isomorphism \cite[Sect. 4, Th.~4.6]{toben}.} 
In particular $\overline{N(\sigma)}$ admits a basic Thom class $[\Phi_{\sigma}]$, of degree $\#\sigma$, that
can be viewed as a basic class on $N$.
Note that the basic Poincar\'e dual class exists, but the identification with the basic Thom class is not established.
Therefore we don't know the behaviour of basic Poincar\'e classes under intersections, so we use basic Thom classes 
instead. We need the following foliation-theoretic result:
\begin{prop}\label{nontrivial_thom}\mbox{} 
Let $\calF$ be a transversely oriented, Killing, Riemannian foliation 
on a compact connected manifold $M$.
If $L$ is a compact leaf, then its basic Thom class is a generator of the top basic cohomology of $M$. 
\end{prop}
\begin{proof}\mbox{}\\
{\em Step 1 -- Preliminaries.} 
In order to prove this statement we need to use the notion of {\em transverse integration} \cite{sergiescu}, 
which in turn involves the so-called {\em Molino bundle} $\hat M$ and {\em basic manifold} $W.$ 
We briefly recall here the notation and main properties of the Molino construction; for 
details we refer to \cite{molino,GT}.
Denote by $q$ the codimension of $\calF$. 
By assumption there is a metric $g$ on the normal bundle $\scalebox{1.5}{$\nu$}\calF$ with respect to which $\calF$ is Riemannian. 
Since $\calF$ is transversely oriented, we can consider the $SO(q)$-principal bundle $\pi\colon{\hat M}\rightarrow M$ 
of positively oriented orthonormal frames of 
$\scalebox{1.5}{$\nu$}\calF=TM/T\calF$.
Let $\omega$ be the transverse Levi-Civita connection
on $\hat M$ and let $H=\ker \omega$ be the corresponding horizontal
distribution.
In particular, at each ${\hat m}\in{\hat M}$ we have the splitting 
$T_{\hat m}{\hat M}=H_{\hat m}\oplus V_{\hat m}$, where 
$V_{\hat m}$ is tangent to the fibre $\pi\inv\big(\pi( \hat m)\big)$. 
The horizontal lift of $\calF$ to $\hat M$ is a transversely parallelizable foliation $\hat{\calF}$ of same dimension, 
Riemannian for a certain metric $\hat g$ on $\scalebox{1.5}{$\nu$}\hat{\calF}=T\hat M/ T\hat{\calF}$.
On each leaf $\hat L$ of $\hat \calF$, $\pi$ is a Galois covering of a corresponding leaf $L$ in $N$ 
(the group acting on $\hat L$ is the holonomy of $L$). 
The {\em commuting sheaf} of germs of 
$\hat\calF$-transverse fields that commute with all global transverse 
fields of $({\hat M},{\hat\calF})$
is locally constant. The foliation $\calF$ is Killing if and only if this sheaf is globally constant. 
In this case the stalk is an abelian Lie algebra $\a$.
This gives rise to two transverse and $\pi$-equivariant actions by $\a$ on $M$ and $\hat M$.
In both $M$ and $\hat M$, the $\a$-orbit of any leaf is its closure, 
which on $\hat M$ has always dimension $\dim\calF+\dim \a$. 
The space of leaf closures $\hat M\big/\overline{{\hat {\calF}}}$ is a smooth manifold $W$ 
called the {\em basic manifold}. The action of $SO(q)$ on $M$ induces a 
smooth action of $SO(q)$ on $W$. The projection $\rho\colon{\hat M}\rightarrow W$ is locally trivial 
and the orbit space $W/SO(q)$
is homeomorphic to 
the space of leaf closures $M\big/\overline{\calF}$. 
Note that if $\calF$ is Killing and transversely oriented then $W$ is orientable \cite[Sect. 5]{toben}.

Denote $l=\dim SO(q)$ and consider a volume element $\dot\nu$ on $\so(q)$. 
Following \cite[Sect. 5]{toben} (see also \cite{sergiescu}) 
we then consider the basic $l$-form
$\nu$ on $\hat M$ defined at each ${\hat m}\in{\hat M}$ by:
$$\nu_{\hat m}(X_1,\ldots,X_l)=\dot\nu(\omega(X_1),\ldots,\omega(X_l)),\quad X_i\in T_{\hat m}{\hat M}.$$ 
The transverse integration of a given basic $q$-form $\alpha$ on $M$ is defined by
$$\int_{\calF}\alpha=\int_{{\hat{\calF}}}(\pi^{*}\alpha)\wedge\nu\text{,}$$
where $\int_{{\hat{\calF}}}$ denotes transverse integration on $\hat M$, which  
can in turn be defined as follows: 
let $\beta\in\Omega^{q+l}_{\hat\calF}(\hat M)$ (in other degrees $\int_{{\hat{\calF}}}$ evaluates to zero). 
Denote by $\iota_{X}\beta$ the contraction of $\beta$ with the fundamental transverse fields $X^{*}_{1},\dots, X^{*}_{\dim \a}$ of the $\a$-action. 
Then $\iota_{X}\beta$ can be written $\rho^{*}(\rho_{\sharp}\beta)$ for some top form 
$\rho_{\sharp}\beta$ on $W$. Now define
$$\int_{\hat\calF} \beta = \int_{W} \rho_{\sharp} \beta.$$
{\em Step 2.}
Now consider the compact leaf $L$ in $M$, and its normal bundle $\scalebox{1.5}{$\nu$}L$.
We consider on $\hat M$ the {\em transverse horizontal bundle} $\calH=H/T\hat\calF$.
The projection $\pi$ induces isomorphisms $H_{\hat m}\simeq T_mM$ and $\calH_{\hat m}\simeq 
T_mM/T_m{\calF} \simeq \scalebox{1.5}{$\nu$}L_{m}$, where $m=\pi({\hat m})$.

According to \cite[Cor. 4.8]{toben}, $\scalebox{1.5}{$\nu$}L$ is oriented {\em as a foliated bundle}.
Therefore, we can consider a basic Thom form 
$\Phi\in \Omega^{q}_{ \calF}(M)$. 
We claim that $\pi^{*}\Phi\in \Omega^{q}(\hat M)$ is $\hat \calF$-basic: 
For a vector $\hat X$ tangent to $\hat\calF$, 
$\iota_{\hat X}\pi^{*}\Phi=\iota_{\pi_{*}\hat X}\Phi \circ \pi_{*}$, but $\iota_{\pi_{*}\hat X}\Phi=0$ as 
$\Phi$ is $\calF$-basic. As $\pi^{*}\Phi$ is closed, we conclude that it is basic with respect to 
$\hat \calF$.

Now, let $\eta$ be the $(q+l-\dim\a)$-form on $W$ such that $\rho^*(\eta)=\iota_X(\pi^*\Phi \wedge\nu)$,~i.e. 
$\eta=\rho_{\sharp}(\pi^*\Phi \wedge\nu)$. The form $\eta$ is a top form on $W$, and by definition of transverse integration 
$$\int_{\calF}\Phi=\int_W \eta.$$
We are left to show that the right hand side is nonzero: by \cite[Sect. 2]{sergiescu}, $\Phi$ is then not exact, thus  
$[\Phi]$ is nonzero in $H^{q}_{\calF}(M)$,
which implies that $H^{q}_{\calF}(M)$ 
is one-dimensional and generated by $[\Phi]$ (\cf (\ref{hom orient}) p.\pageref{hom orient} 
and comments below).
Let ${\hat m}\in{\hat M}$. Since $\calH_{\hat m}\simeq T_mM/T_m{\calF}$ 
and $\nu$ is zero on horizontal vectors, the 
$\hat \calF$-basic top form $\pi^{*}\Phi \wedge\nu$ is nonzero at $\hat m$ is and only if 
$\pi^{*}\Phi$ is. 
Since $L$ is connected, we can assume that the subset 
$\{{\hat m}\in{\hat M}\;|\;(\pi^{*}\Phi)_{\hat m}\neq0\}$ is open, 
connected, and $\overline{\hat \calF}$-saturated, since $\pi^{*}\Phi$ is $\hat\calF$-basic and therefore $\a$-invariant 
\cite[Lem.~3.15]{GT}.
The top form $\eta$ is non zero at $w\in W$ if and only if
$w$ lies in the open, connected subset 
$\rho(\{{\hat m}\in{\hat M}\;|\;(\pi^{*}\Phi)_{\hat m}\neq0\})$. It follows that
$\int_W\eta\neq0$.
\end{proof}

\begin{thm}\label{generated in dim two}
Let $(V,\bT)$ be a polytopal, balanced and odd triangulated vector configuration. Let $(N,\calF)$ be any foliated manifold built from $(V,\bT)$. Then the basic cohomology algebra $H^{\bullet}_{\calF}(N)$ is generated by classes of degree two.
\end{thm}

\begin{proof}
Fix $r$ such that $0\leq r\leq d$. By Th.~\ref{prop betti numbers}, $\dim b_{\calF}^{2r}=h_r$. Moreover, from
the proof of Th.~\ref{prop betti numbers}, 
there are, in the shelling $\E_1,\ldots,\E_{f_{d-1}}$, exactly $h_r$ maximal simplices $\E_{\alpha_1},\ldots,\E_{\alpha_{h_r}}$ whose restrictions 
$\tau_{\alpha_j}$'s have cardinality equal to $r$.

For every $j$, denote the basic Thom classes of  $\overline{N(\tau_{\alpha_j})}$ and $\overline{N(\tau_{\alpha_j}^-)}$  
by $[\psi_{j}]$ and $[\psi_j^-]$, of degree $2r$ and $2d-2r$ respectively.

We prove below that the classes $[\psi_j]$, $j=1,\ldots,h_r$ are linearly independent, 
and therefore give a basis of the vector space $H_{\calF}^{2r}(N)$. Consider a linear combination 
$\sum_{j=1}^{h_{r}}a_j[\psi_j]=0$, with $a_j\in\R$.

By Lem. \ref{Lemma_intersections} {\it (i)},  the cup product 
$[\psi_1^{-}] \smile [\psi_j]$ is zero when $j>1$ as these forms have disjoint supports, 
so the cup product of the above equality with $[\psi_1^{-}]$ 
gives $a_{1} [\psi_1^{-}] \smile [\psi_1]=0$.

The classical proof \cite{bott-tu} can be adapted
to show that given two transversely intersecting, smooth, compact, saturated,  
submanifolds, with oriented foliated normal bundles, 
the basic Thom class of their intersection is the cup product of their basic Thom classes. 
We know that $\calF$ is Riemannian and Killing (\cf \ref{F is Riemannian}). 
Therefore, by Lem. \ref{Lemma_intersections} {\it (ii)} and Prop.~\ref{nontrivial_thom}, 
$[\psi_1^{-}] \smile [\psi_1]$ 
is a nonzero generator of the top basic cohomology of $(M,\calF)$, which is one-dimensional 
by (\ref{hom orient}) below. 
Then $a_{1}$ must be zero, and repeating this argument shows that all coefficients vanish. 

Now, for each index $i$ in $\{1,\ldots,h\}$, 
the basic Thom class
of $\overline{N(\{i\})}$ is a basic class $[\varphi_i]$ of degree $2$ on $N$. 
Remarking that for each $j$, $\overline{N(\tau_{\alpha_j})}$ is the transverse intersection
$\cap_{i\in\tau_{\alpha_j}}\overline{N(\{i\})}$
(an analogous idea is used in the proof of \cite[Prop.~3.10 (ii)]{DJ}), we conclude that
\rev{$[\psi_j]={\displaystyle \smile_{i\in\tau_{\alpha_j}}} [\varphi_i]$}. 
It follows that the algebra
 $H^{\bullet}_{\calF}(N)$ is generated in degree $2$.
\end{proof}
\section{Examples}\label{examples}
\subsection{The projective line $\C\P^{1}$ and variants} Let $E=\R$. Consider the
configuration $V:=\big(p, -q , q-p, 0\big)$, where $p$ and $q$ are positive reals. \rev{We have $d=1$ and $n=4$, so $m=1$.} \rev{As noticed in \ref{CmandN}, we can use an ambient automorphism to normalize $V$} to
$\big(\frac{p}{q}, -1 , 1-\frac{p}{q}, 0\big)$. 
Therefore, there is only one real parameter here, namely the fraction 
$\frac{p}{q}$. We will distinguish the following cases
\begin{itemize}
\item[(a)] $\frac{p}{q}\in\Q$  with $p,q$ coprime integers;
\item[(b)] $\frac{p}{q}\in\R\setminus\Q$. 
\end{itemize}
We triangulate $V$ by $\bT=\{\E_{1}=\{1\},\E_{2}=\{2\},\varnothing\}$
\rev{---in particular $v_{3}= 1-\frac{p}{q}$ and $v_{4}=0$ are ghosts}.  
Then the fan and quasilattice associated with $(V,\bT)$ are respectively: 
the one dimensional fan whose maximal cones are
$\text{cone}(\E_{1})=\text{Span}_{\R_{\geq 0}}v_{1}=\R_{\geq 0}$ and 
$\text{cone}(\E_{2})=\text{Span}_{\R_{\geq 0}}v_{2}=\R_{\leq 0}$, that is the usual fan of $\C P^1$; 
the quasilattice $Q=\textrm{Span}_{\Z}\{1,\frac{p}{q}\}$. In case (a) $Q$ is $\Z$. In case (b), $Q$ is dense and has rank two.
The virtual chamber is 
$\bE=\{\E^{c}_{1}=\{234\}, \E^{c}_{2}=\{134\}\}$ and 
\begin{align*}
U(\bT) & =U_{1} \cup U_{2} \\
          & =\setofst{[\vz]\in \C\P^{3}}{z_{2}\neq 0, z_{3}\neq 0, z_{4}\neq 0}
\cup \setofst{[\vz]\in \C\P^{3}}{z_{1}\neq 0, z_{3}\neq 0, z_{4}\neq 0}.
\end{align*} 
Choose 
$M=\begin{bmatrix} 1&1&0\\1&\frac{p}{q}&0\\1&0&0\\1&0&1\\ \end{bmatrix}$
so $\Lambda^{\R}=\begin{bmatrix} 1&\frac{p}{q}&0&0\\0&0&0&1\\ \end{bmatrix}$ 
and $\Lambda=\begin{bmatrix} 1&\frac{p}{q}&0&i \end{bmatrix}$. 
\rev{It is now straightforward to write explicitly action  (\ref{action giving N}) and action  (\ref{action giving X}).}
The leaf $\calF_{1}$ corresponding to the simplex $\{1\}$ is given by
$\C_{\calF}/\pi(L_1) \approx  \C/ \text{Span}_{\Z}\{\frac{q}{p},i\}$, while
the leaf $\calF_{2}$ corresponding to the simplex $\{2\}$ is given by
$\C_{\calF}/\pi(L_2) \approx  \C/\text{Span}_{\Z}\{1,i\}$.
Now let $[\vz]$ be a generic point, i.e., $[\vz]\in N(\varnothing)$. 
Then $(t,u)\in L_{\vz}\Leftrightarrow qt,pt,u\in\Z$. In case (a), $L_{\vz}=\Z^{2}$, so 
$\calF_{\vz}\approx \C/\text{Span}_{\Z}\{1,i\}$. 
 
In case (b), $\textrm{Rank}(L_{\vz})=0$, the generic leaf is $\C$ and its closure in $N$ is an $(S^1)^3$.  

From the proof of Lemma \ref{lemma-sphere}, we know that $\C\hookrightarrow U_{1}, z_{1}\mapsto [z_{1}:1:1:1]$ 
gives a local slice for action (\ref{action giving X}). 
This slice is stabilized by \rev{$\frac{q}{p}\Z\times\Z\subset\C^{2}$}. Hence,  
$X_{1}:=N_{1}/\calF=U_{1}/\C^{2}$ can be identified with  $\C$ modulo \rev{$z_{1}\mapsto \expo{\frac{q}{p}}z_{1}$.} 
\rev{Similarly $X_{2}:=N_{2}/\calF=U_{2}/\C^{2}$ can be identified with  $\C$ modulo $z_{2}\mapsto \expo{\frac{p}{q}t}z_{2}$.
The leaf space is then $X=X_{1}\cup X_{2}$.}
In case (a), if $p=q=1$ the leaf space $X$ is $\C P^1$. 
For $p,q$ any coprime integers, the leaf space
is a weighted projective space, that is the quotient of $\C^2\setminus\{0\}$
by the action of $\C^*$ with weights $q$ and $p$.

In case (b), the local groups at the poles are infinite, of rank one. 
The leaf space is the toric quasifold described in detail in \cite[Ex.~1.13,3.5]{p} and \cite[Ex.~2.6]{cx}.

To describe $\calF$, and see how it desingularizes $X$, compose the above $\C\hookrightarrow U_{1}$ with the 
quotient $U_{1}\to N_{1}$. We obtain a slice $\C\hookrightarrow N_{1}$ for the action of $\C_{\calF}$. 
The leaf passing through $z_{1}=0$ is the leaf $\calF_{1}$ above; it intersects the slice only once. 
The leaf through any $z_{1}\neq 0$ hits the slice $p$ times \rev{in case (a) and}
infinitely many times \rev{in case (b).}
\rev{Hence it wraps around $\calF_{1}$ $p$ times or infinitely many times
respectively.}
 
 \subsection{The non-necessarily reduced orbifold $\C\P^{1}$}\label{arbitrary p q}
Now we encode in a vector configuration the case of $p$ and $q$ non necessarily coprime, $Q=\Z$.
Choose $V=(p,-q,1,q-p-1)$ and take 
$M=\begin{bmatrix} 1&q&0\\1&p&1\\1&0&q\\1&0&0\\ \end{bmatrix}$
so $\Lambda^{\R}=\begin{bmatrix} q&p&0&0\\0&1&q&0\\ \end{bmatrix}$ 
and $\Lambda=\begin{bmatrix} q&p+i&qi&0 \end{bmatrix}$. 
Action (\ref{action giving N}) becomes 
$t.[\vz]=[e^{2\pi i q t} z_1: e^{2\pi i (p+i) t} z_{2}: e^{-2\pi  q t} z_{3}:  z_{4} ]$ 
and action (\ref{action giving X}) becomes 
$(t,u). [\vz] = [e^{2\pi i q t} z_1: e^{2\pi i (p t+u)} z_{2}:  e^{2\pi i qu} z_{3}:  z_{4} ]$.

We take the same triangulation $\bT$ as above, so we can use the same slices,
which are stabilized respectively by $\big\{(t,u)=(\frac{qk-l}{pq},\frac{l}{q})\ \big\vert\  k,l\in\Z\big\}\subset \C^{2}$ 
and $\big\{(t,u)=(\frac{k}{q},\frac{l}{q})\ \big\vert\ k,l\in\Z\big\}\subset \C^{2}$. 
These groups act on the slices by $(k,l).z_{1}= 
\expo{qt}z_{1}=\expo{\frac{qk-l}{p}}z_{1}$ and 
$(k,l).z_{2}= \expo{(pt+u)} z_{2}= \expo{\frac{pk+l}{q}} z_{2}$. 
The leaf space is therefore an orbifold with singularities at the poles of order $p,q\in \Z_{\geq 1}$.
This is similar to \cite[Ex. 5.3]{MV}, 
although our construction does not involve the choice of a K\"ahler class 
(we give an interpretation of this extra piece of information in \ref{Example height fct}). 
In conclusion, in the rational case, one can prescribe at the poles orbifold singularities of arbitrary order
$p,q$, with $p,q\in \Z_{\geq 1}$.
\rev{
Referring to Rem.~\ref{downstairs foliation}, consider the case $\gcd(p,q)=e>1$ and
let $a,b\in\Z$ such that $ap+bq=e$. Then  
the leaf space is $X=U\rev{'}(\Delta)/G$, where $U\rev{'}(\Delta)=\C^2\setminus\{0\}$ and $G=\C\times\frac{\Z}{e\Z}$
acts on $U\rev{'}(\Delta)$ by 
$(t,[n]).(z_{1},z_{2})=(e^{2\pi i ( q t+\frac{a}{e}n) } z_1, e^{2\pi i (pt-\frac{b}{e}n)} z_{2}).$
}

\subsection{The nonrational $\C\P^{1}$}\label{Nonrational CP1}\mbox{} 
By a suitable choice of a vector configuration \rev{---in particular of the ghosts vectors---} one can prescribe 
an {\em arbitrary} finitely generated subgroup $A$ of the circle as local group at both poles
of the corresponding toric quasifold. 
Assume without loss of generality that $A$ is generated by 
$\expo{r_{1}},\dots,\expo{r_{2m-1}}$ with $r_j\in\R$.
Let $V=(1,-1,r_{1},\dots, r_{2m-1},-r_{1}-\dots-r_{2m-1})$. Then the quasilattice 
$Q$ is $\textrm{Span}_{\Z}\{1,r_{1},\dots, r_{2m-1}\}$. 
\rev{Keep $\bT$ as above, so $v_{3}\dots v_{2m+2}$ are ghosts}. 
Take $$M=\begin{bmatrix} 
1& 1&-r_{1}&\dots&\dots &-r_{2m-1}\\
1&1 &0&\dots&\dots&0\\
1&0&1&0&\dots&0\\
\vdots&\vdots&\ddots&\ddots&\ddots&\vdots\\
\vdots&\vdots&\vdots&\ddots&\ddots&0\\
1&0&0&\dots&0&1\\
1&0&0&\dots&\dots&0\\
 \end{bmatrix}\text{,}$$
so $\Lambda^{\R}=
\Big(\begin{bmatrix}1\\-r_{1}\\ \vdots \\ -r_{2m-1}  \end{bmatrix}, e_{1}, \dots, e_{2m}, 0  \Big)$, 
where  $e_{1}, \dots, e_{2m}$ is the canonical basis of $\R^{2m}$.
The slice $\C\hookrightarrow U_{1}, 
z_{1}\mapsto[z_{1}:1:\dots:1]$ is stabilized by $\Z^{2m-1}\subset\C^{2m}$, 
which acts  by 
$\underline{h}.z_{1}= \expo{(r_{1}h_{2}+\cdots+r_{2m-1}h_{2m})}z_{1}$.
At the other pole the local group is also $A$, which acts on the corresponding slice in the same way.

\subsection{Stanley's proof in our setting}\label{Stanley}

\noindent

\noindent Let $\Delta$ be a polytopal simplicial fan. 
Denote by $(h_{0},\dots, h_{d})$ its $h$-vector (\cf Sect. \ref{shellings and h}). 
Define its {\em $g$-vector} $(g_1,\ldots,g_{\delta})$, 
with $\delta=[\frac{d}{2}]$, by $g_j=h_j-h_{j-1}$, $j=1\ldots\delta$. 
Choose a triangulated vector configuration $(V,\bT)$ whose associated fan is $\Delta$, 
and a corresponding $(N,\calF)$ (\cf Sect. \ref{construction of N}). 

In close analogy to \cite{S}, we show below that the combinatorial properties that characterize 
the $g$-vector of a simple polytope have a direct interpretation, and proof, 
in terms of the basic cohomology of $(N,\calF)$.

We first remark that by Th.~\ref{prop betti numbers}, $\calF$ is homologically 
orientable, i.e.
\begin{equation}\label{hom orient}
H_{\calF}^{2d}(N)\neq 0\text{,}
\end{equation}
which is equivalent to either $H_{\calF}^{2d}(N)=1$, or 
to Poincar\'e duality for the basic cohomology of $(N,\calF)$ \cite{elkacimihector,haefliger,sergiescu}.

\subsubsection{Dehn-Sommerville equations} 
By our computation of the basic Betti numbers in Th.~\ref{prop betti numbers},
the Dehn-Sommerville equations
$$h_{d-j}=h_{j}\quad\text{for all } j$$
are only a restatement of basic Poincar\'e duality.

As is well-known, these relations follow more simply by computing the $h_{i}$'s  
using a shelling and the reverse shelling \cite[8.21]{ziegler}.

\subsubsection{Nonnegativity of the $g$-vector}\label{positivity}
\rev{The foliation $\calF$ is transversely K\"ahler by Loeb-Nicolau (for $m=1$, in \cite{LN}) and 
Meersseman (for $m\geq 1$, in \cite{M}). By (\ref{hom orient}) and
\cite[3.4.7]{EK}, $H^{\bullet}_{\calF}(N)$ has the hard Lefschetz property. 
In particular there exists an injective map 
$L:H_{\calF}^{2j-2}(N) \to H_{\calF}^{2j}(N)$ for all $j\leq \delta$.
Therefore $b_{\calF}^{2j}(N)-b_{\calF}^{2j-2}(N)\geq0$. By Th.~\ref{prop betti numbers}, $g_j=h_j-h_{j-1}\geq0$ 
for $j=1\ldots\delta$.
}
\subsubsection{Bound on growth of $g_j$}
\noindent
The usual numerical condition (\cf \cite[p.127 II(b)]{Ful}) is equivalent to the existence of a graded commutative 
algebra $R=R_{0}\oplus R_{1}\oplus \dots \oplus R_{\delta}$ over the field $K=R_{0}$, generated by $R_{1}$, and such 
that $g_{j}=\dim R_{j}$ for $j=1\dots\delta$. 
We take $R_{j}:=H^{2j}_{\calF}(N)/ L(H^{2j-2}_{\calF})$. The result follows from the hard Lefschetz property and Th.~\ref{generated in dim two}.

\rev{\begin{rem}{\rm This shows that at least part of the technology available to toric varieties survives to the nonrational case. However, in order to prove Stanley's result it is possible to  bypass toric geometry:
we refer the reader to \cite{FK} for the latest in a series of results initiated by McMullen \cite{McM} 
and continued by Timorin \cite{Tim} and others.
}
\end{rem}}

\subsection{Brief account of the polytopal case} \label{Reg Triang and Kahler}
\subsubsection{Preliminaries}\label{preliminaries}
An important special case is when $\bT$ is a {\em regular triangulation}. 
Regularity has several characterisations: 
\begin{enumerate}
\item[\phantom{$\Leftrightarrow$}\quad 1.] The triangulation is {\em regular}
\item[$\Leftrightarrow$\quad  2.]  The fan $\Delta$ is {\em polytopal}
\item[$\Leftrightarrow$\quad  3.] There exists a {\em height function} on $V$ that induces $\bT$
\item[$\Leftrightarrow$\quad 4.] The virtual chamber defines a nonempty {\em chamber}, i.e.,  
$\bigcap_{\alpha} \mathring{P}_\alpha \neq\varnothing$  (\cf \ref{Proof that Bosio's conditions hold})
\item[$\Leftrightarrow$\quad 5.] 
There exists $\nu\in \R^{2m}$  s.t. $\forall \tau\subset \{1\dots n\}, \tau\in \bT$ if and only if $\nu$ is 
in the interior of the convex hull of  $\setofst{\Lambda^{\R}_{j}}{j\in\tau^{c}}$
\end{enumerate}
The last condition implies that, by definition, the corresponding manifold $N$ is an LVM-manifold \cite{M}.
This in turn implies  that the foliation $\calF$ is transversely K\"ahler by 
\cite{LN} (for $m=1$) and \cite{M} (for $m\geq 2$).

\subsubsection*{Correspondence between regular triangulations and chambers} 

Regular triangulations of $V$ are in one-to-one correspondence with {\em chambers} of $\Lambda^{\R}$, i.e., 
bounded connected components of $\R^{2m}-L$, where $L$ is the union of all affine 
$2m-1$-planes determined by $\Lambda^{\R}_{1},\dots,\Lambda^{\R}_{n}$. 
Explicitly: from $\bT$ we obtain the chamber $\bigcap_{\alpha} \mathring{P}_\alpha$; from a chamber $C$, we define $\bT$ by $\tau\in \bT \Leftrightarrow$ the convex hull of $\setofst{\Lambda^{\R}_{j}}{j\in\tau^{c}}$ contains $C$.  

\subsubsection*{Correspondence between height functions and points in a chamber} 
A triangulation is {\em regular} when there exists a so-called {\em height function} 
$\omega=(\omega_{1},\dots,\omega_{n})\in \R^{n}$ such that: 
there is a (necessarily unique) convex function $\psi_{\omega}:\R^{d}\to \R$, restricting to pairwise distinct linear forms on the maximal cones of $\Delta$, 
such that $\psi_{\omega}(v_{i})=\omega_{i}$ for each non ghost vector $v_{i}$ 
and $\psi_{\omega}(v_{i})<\omega_{i}$ for each ghost vector $v_{i}$. 

By a result of Carl Lee \cite[Lemma 5.4.4]{DL-R-S}, $\omega$ induces a triangulation 
$\bT$ if and only if 
$\nu := \frac{1}{\sum \omega_{i}}\sum \omega_{i} {\Lambda}_{i}^{\R}$
belongs to the chamber associated to $\bT$ as above.
Therefore, conversely, starting from a point in a chamber written as a convex linear combination $\sum \omega_{i} \Lambda_{i}^{\R}$, we obtain a height function $\omega$ inducing a regular triangulation.

\medskip \noindent
{\em Conclusion: the map $\omega \mapsto \nu$
gives a quantitative refinement of the qualitative 
correspondence between regular triangulations and chambers described above.}

\subsubsection{Example}\label{Example height fct}
We choose the same data as in \ref{arbitrary p q} (but here $p,q$ can be any positive reals): 
$V=(p,-q,1,q-p-1)$ and $\Lambda^{\R}=\begin{bmatrix} q&p&0&0\\0&1&q&0\\ \end{bmatrix}$ 
\begin{center}
\def\svgwidth{0.9\textwidth}
 %% Creator: Inkscape inkscape 0.48.5, www.inkscape.org
%% Accompanies image file '2014-11-11_CP1_V_and_Lambda.pdf' (pdf, eps, ps)
\begingroup%
  \makeatletter%
  \providecommand\color[2][]{%
    \errmessage{(Inkscape) Color is used for the text in Inkscape, but the package 'color.sty' is not loaded}%
    \renewcommand\color[2][]{}%
  }%
  \providecommand\transparent[1]{%
    \errmessage{(Inkscape) Transparency is used (non-zero) for the text in Inkscape, but the package 'transparent.sty' is not loaded}%
    \renewcommand\transparent[1]{}%
  }%
  \providecommand\rotatebox[2]{#2}%
  \ifx\svgwidth\undefined%
    \setlength{\unitlength}{551.92998047bp}%
    \ifx\svgscale\undefined%
      \relax%
    \else%
      \setlength{\unitlength}{\unitlength * \real{\svgscale}}%
    \fi%
  \else%
    \setlength{\unitlength}{\svgwidth}%
  \fi%
  \global\let\svgwidth\undefined%
  \global\let\svgscale\undefined%
  \makeatother%
  \begin{picture}(1,0.23279946)%
    \put(0,0){\includegraphics[width=\unitlength]{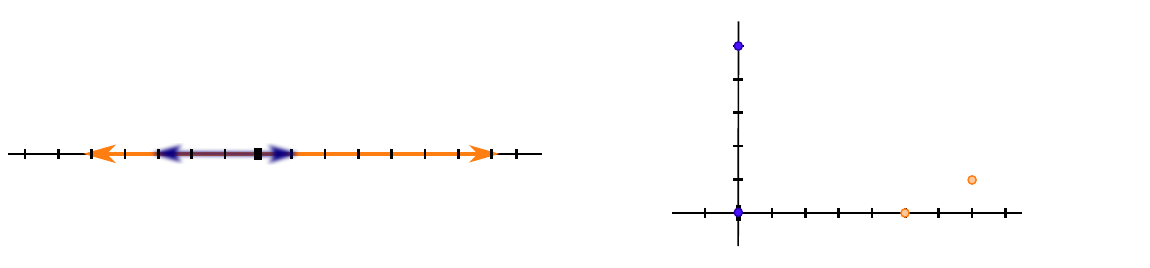}}%
    \put(0.16004953,0.07996065){\color[rgb]{0,0,0}\makebox(0,0)[lb]{\smash{
}}}%
    \put(0.41545673,0.07168279){\color[rgb]{0,0,0}\makebox(0,0)[lb]{\smash{$v_1$}}}%
    \put(0.24446314,0.07168279){\color[rgb]{0,0,0}\makebox(0,0)[lb]{\smash{$v_3$}}}%
    \put(0.12831655,0.07168279){\color[rgb]{0,0,0}\makebox(0,0)[lb]{\smash{$v_4$}}}%
    \put(0.0710498,0.07168279){\color[rgb]{0,0,0}\makebox(0,0)[lb]{\smash{$v_2$}}}%
    \put(0.65063716,0.20389326){\color[rgb]{0,0,0}\makebox(0,0)[lb]{\smash{$\Lambda_{3}^{\R}$}}}%
    \put(0.80676047,0.09395398){\color[rgb]{0,0,0}\makebox(0,0)[lb]{\smash{$\Lambda_{2}^{\R}$}}}%
    \put(0.7892708,0.0120805){\color[rgb]{0,0,0}\makebox(0,0)[lb]{\smash{$\Lambda_{1}^{\R}$}}}%
    \put(0.64854391,0.0120805){\color[rgb]{0,0,0}\makebox(0,0)[lb]{\smash{$\Lambda_{4}^{\R}$}}}%
    \put(0.21892925,0.11920123){\color[rgb]{0,0,0}\makebox(0,0)[lb]{\smash{$0$}}}%
  \end{picture}%
\endgroup%

 \end{center}
To induce the triangulation $\bT=\{\E_{1}=\{1\},\E_{2}=\{2\},\varnothing\}$, we can choose, for example,  
$\omega_{1}=p$ , $\omega_{2}=q$ (so $\psi_{\omega}=|.|$),
$\omega_{3}>1$ and $\omega_{4}>|q-p-1|$:
$$\includegraphics[width=0.5\textwidth]{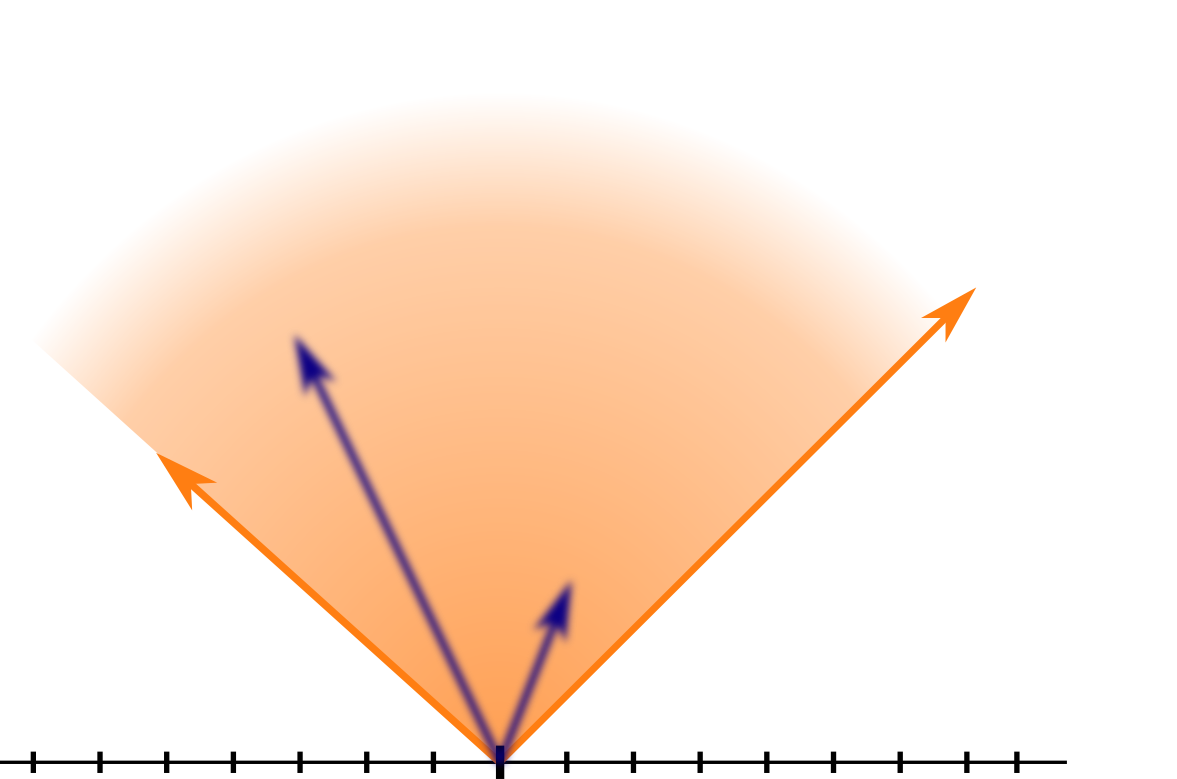}$$
which in turns gives the point $\nu= 
\frac{1}{\omega_{1}+\omega_{2}+\omega_{3}+\omega_{4}} 
(\omega_{1}\Lambda_{1}^{\R}+\omega_{2} \Lambda_{2}^{\R}+\omega_{3} 
\Lambda_{3}^{\R}+\omega_{4} \Lambda_{4}^{\R})$ contained in one of the four chambers of the configuration $\Lambda^{\R}$:
\begin{center}
\def\svgwidth{0.5\textwidth}
 %% Creator: Inkscape inkscape 0.48.5, www.inkscape.org
%% Accompanies image file '2013-12-28_CP1_Lambda_chamber_2.pdf' (pdf, eps, ps)
\begingroup%
  \makeatletter%
  \providecommand\color[2][]{%
    \errmessage{(Inkscape) Color is used for the text in Inkscape, but the package 'color.sty' is not loaded}%
    \renewcommand\color[2][]{}%
  }%
  \providecommand\transparent[1]{%
    \errmessage{(Inkscape) Transparency is used (non-zero) for the text in Inkscape, but the package 'transparent.sty' is not loaded}%
    \renewcommand\transparent[1]{}%
  }%
  \providecommand\rotatebox[2]{#2}%
  \ifx\svgwidth\undefined%
    \setlength{\unitlength}{184.59343262bp}%
    \ifx\svgscale\undefined%
      \relax%
    \else%
      \setlength{\unitlength}{\unitlength * \real{\svgscale}}%
    \fi%
  \else%
    \setlength{\unitlength}{\svgwidth}%
  \fi%
  \global\let\svgwidth\undefined%
  \global\let\svgscale\undefined%
  \makeatother%
  \begin{picture}(1,0.53835078)%
    \put(0,0){\includegraphics[width=\unitlength]{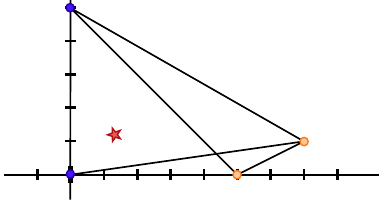}}%
    \put(0.25187652,0.21132111){\color[rgb]{0,0,0}\makebox(0,0)[lb]{\smash{$\nu$}}}%
  \end{picture}%
\endgroup%

 \end{center}
Using \cite{M}, this point can be used to give a {\em $\calC^{\infty}$  embedding} $N\hookrightarrow\C\P^{n-1}$ as 
$$\mathcal{N}=\bigg\{[\vz]\in\C\P^{n-1}\ \bigg\vert
\ \sum_{j=1\dots n} (\Lambda_{j}^{\R} -\nu)\, |z_{j}|^{2}=0\bigg\}.$$
This solves a problem mentioned in \cite[Rem. 4.11]{MV}.
Pulling-back the Fubini-Study K\"ahler form by this embedding endows $N$ with a $2$-form $\varphi$ 
transversely K\"ahler with respect to $\calF$ \cite{MV}.
The form $\varphi$ defines on $X$ a K\"ahler form, whose moment polytope is 
$[-\frac{1}{p+q+\omega_3+\omega_4},\frac{1}{p+q+\omega_3+\omega_4}]$.

\bigskip

\section*{Concluding remarks}
\begin{itemize}
\item Using the above notation, notice that we can also interpret the choice of $\nu$ on the boundary of a chamber: this corresponds to a non simplicial polyhedral decomposition of the vector configuration $V$ (\cf \cite{DL-R-S}). We intend to use our methods to investigate this case. 
\item Starting from a rational triangulated vector configuration, we can rescale each vector to obtain 
a nonrational configuration, which is in a sense ``weakly'' nonrational (the fan is still rational). It would be interesting to characterize geometrically the class of foliated manifolds obtained from such configurations. 
\item In order to prove Th. \ref{prop betti numbers} we assume that our simplicial fan is shellable. It is an open problem to decide if all simplicial fans are shellable (specialists we consulted lean toward a negative answer). It seems possible to prove Th. \ref{prop betti numbers} in general by constructing an ad hoc spectral sequence. 
\new{\item We expect the basic cohomology ring to have a similar description to the real cohomology ring
of simplicial toric varieties.} 
\item We conjecture that the basic Hodge numbers of these foliations are concentrated on the diagonal.
\item Using the result of Th. \ref{prop betti numbers}, we hope to be able to prove the following result 
(conjectured in \cite{CZ}): an LVMB-manifold is an LVM-manifold if and only if the foliation $\calF$ is transversely K\"ahler.
\end{itemize}

\bigskip

{\small 

Fiammetta Battaglia\\
%\noindent 
Universit\`a di Firenze, Dipartimento di Matematica e Informatica U. Dini, 
Via S. Marta 3, 50139 Firenze, Italy.

\noindent
e-mail: fiammetta.battaglia@unifi.it}

\medskip

{\small 

Dan Zaffran\\
%\noindent
Florida Institute of Technology, Dept. of Mathematical Sciences,150 W. University Blvd.
Melbourne, FL 32901, U.S.A.\\
Korea Advanced Institute of Science and Technology, 291 Daehak-ro, Yuseong-gu, Daejeon 305-338, South Korea.

\noindent
e-mail: dzaffran@fit.edu}

\enlargethispage{10mm}

\end{document}